\documentclass[10pt,reqno]{amsart}

\usepackage{amssymb}
\usepackage{amsfonts}
\usepackage{amsthm}
\usepackage{amsmath}
\usepackage{bbm}

\numberwithin{equation}{section}
\allowdisplaybreaks

\newtheorem{theorem}{Theorem}[section]
\newtheorem{proposition}[theorem]{Proposition}

\newtheorem{lemma}[theorem]{Lemma}
\newtheorem{corollary}[theorem]{Corollary}

\theoremstyle{definition}

\theoremstyle{remark}

\newcommand{\R}{\mathbb{R}}

\newcommand{\Z}{\mathbb{Z}}

\renewcommand{\hat}{\widehat}
\newcommand{\eps}{\varepsilon}
\newcommand{\nmin}{n_{\rm{min}}}
\newcommand{\nmax}{n_{\rm{max}}}
\newcommand{\Imin}{I_{\rm{min}}}
\newcommand{\Imed}{I_{\rm{med}}}
\newcommand{\Imax}{I_{\rm{max}}}
\newcommand{\Amin}{A_{\rm{min}}}
\newcommand{\Amed}{A_{\rm{med}}}
\newcommand{\Amax}{A_{\rm{max}}}

\newcommand{\scriptB}{\mathcal{B}}

\newcommand{\scriptE}{\mathcal{E}}
\newcommand{\scriptF}{\mathcal{F}}

\newcommand{\scriptQ}{\mathcal{Q}}
\newcommand{\scriptR}{\mathcal{R}}
\newcommand{\scriptS}{\mathcal{S}}
\newcommand{\scriptT}{\mathcal{T}}

\newcommand{\ctc}[1]{\,\,\text{#1}\,\,}

\DeclareMathOperator*{\supp}{supp}

\DeclareMathOperator*{\dist}{dist}

\begin{document}

\title[Restriction to rotationally symmetric hypersurfaces]{Linear and bilinear restriction to certain rotationally symmetric hypersurfaces\footnote{2010 MSC 42}}

\author{Betsy Stovall}
\address{Department of Mathematics, University of Wisconsin, Madison, WI 53706}
\email{stovall@math.wisc.edu}

\begin{abstract}  Conditional on Fourier restriction estimates for elliptic hypersurfaces, we prove optimal restriction estimates for polynomial hypersurfaces of revolution for which the defining polynomial has non-negative coefficients.  In particular, we obtain uniform--depending only on the dimension and polynomial degree--estimates for restriction with affine surface measure, slightly beyond the bilinear range. The main step in the proof of our linear result is an (unconditional) bilinear adjoint restriction estimate for pieces at different scales.
\end{abstract}

\maketitle

\section{Introduction} \label{S:intro}

Recently, there has been considerable interest (e.g.\ \cite{BakIJM94, Buschenhenke, BMV, CKZ1, CKZ2, IM, OberlinPAMS04, Sjolin}) in extending the restriction problem to degenerate hypersurfaces, that is, hypersurfaces for which one or more of the principal curvatures is allowed to vanish to some finite (or infinite) order.  It has been known for a number of years that if the hypersurface is equipped with Euclidean surface measure, the exponent pairs for which restriction phenomena are possible must depend on the `type,' or order of vanishing of the curvatures.  Affine surface measure, however, is conjectured to mitigate the effects of such degeneracies and allow for restriction theorems that are uniform over large classes of hypersurfaces.  We verify this conjecture for a class of rotationally symmetric hypersurfaces by proving that the elliptic restriction conjecture implies the restriction conjecture with affine surface measure.

Consider the hypersurface
$$
\Gamma = \{(G(\xi),\xi) : \xi \in U \subseteq \R^d\}.
$$
We say that $\Gamma$ (or $G$) is elliptic with parameters $A$, $N$, and $1 > \epsilon_0 > 0$ if $U$ is a subset of the unit ball $B$, $\|\nabla G\|_{C^N(B)} \leq A$, and the eigenvalues of $D^2G(x)$ lie in $(\epsilon_0,\epsilon_0^{-1})$ for all $x \in U$.  

The restriction conjecture for elliptic hypersurfaces is the statement that for all pairs $(p,q)$ satisfying the (restriction) admissibility condition
\begin{equation}\label{E:admissible}
\tfrac{2(d+1)}{d+2} < q \leq \infty, \qquad  q = \tfrac{dp'}{d+2},
\end{equation}
there exists $N=N_p$ such that for all parameters $A,\epsilon_0$, and all elliptic phases $\Phi$ with parameters $A,N,\epsilon_0$,
\begin{equation} \label{E:elliptic restriction} 
\bigl(\int_B |\hat f(\Phi(\xi),\xi)|^q\, d\xi\bigr)^\frac1q \lesssim \|f\|_{L^p_{t,x}(\R^{1+d})} \qquad f \in \scriptS(\R^{1+d}),
\end{equation}
where $\scriptS$ denotes the Schwartz class and the implicit constant is allowed to depend on $p,A,\epsilon_0$.  We let $\scriptR(p \to q)$ denote the statement that the restriction conjecture for elliptic hypersurfaces is valid for the exponents $p,q$.  We note that our definition of elliptic differs slightly from that in \cite{TVV}, but by a well-known argument (a partition of unity coupled with affine transformations), the corresponding restriction conjectures are easily seen to be equivalent.

In the notation above, affine surface measure on $\Gamma$ is the pushforward by $\xi \mapsto (G(\xi),\xi)$ of  
$$
\Lambda_G(\xi)\, d\xi := |\det D^2 G(\xi)|^{\frac1{d+2}}\, d\xi;
$$
more geometrically, for $\omega \in \Gamma$, it equals $|\kappa(\omega)|^{\frac1{d+2}}\, d\sigma(\omega)$, where $\kappa$ is the Gaussian curvature and $d\sigma$ is Lebesgue measure on $\Gamma$  \cite{OberlinHyp00}. Since this measure gives little weight to the `bad' flat regions of $\Gamma$, it is natural to ask whether it is possible to prove restriction estimates of the form
$$
\|\hat f (G(\xi),\xi) \|_{L^p_\xi(\R^{d};\Lambda_G)} \lesssim \|f\|_{L^q_{t,x}(\R^{1+d})}, \qquad f \in \scriptS,
$$
for $(p,q)$ satisfying the admissibility condition \eqref{E:admissible} and with the implicit constant uniform over $G$ in some reasonably large class.  Oscillation is a well-known enemy of restriction estimates--consider, for instance Sj\"olin's counter-example $(t,\sin(t^{-k})e^{-1/t})$ \cite{Sjolin}\footnote{Examples in higher dimensions may be found in \cite{CZ}.}--so it is natural to consider the affine restriction problem for $G$ a polynomial of bounded degree.  

Here we specialize somewhat more.  Let $P$ be an even polynomial on $\R$ with non-negative coefficients, and let 
$$
S_P = \{(P(|\xi|),\xi) : \xi \in \R^d\}.
$$
The following is our main result.  

\begin{theorem}\label{T:main}
Assume that the restriction conjecture $\mathcal R(p_0 \to q_0)$ holds for some admissible pair of exponents.  Then for every restriction admissible pair $(p,q)$ with $p < p_0$, if $P:\R \to \R$ is an even polynomial of degree $N$ with non-negative coefficients, the restriction estimate
\begin{equation} \label{E:restriction}
\bigl(\int |\hat f(P(|\xi|),\xi)|^q \, \Lambda_P(\xi)\, d\xi\bigr)^{\frac1q} \leq C \|f\|_{L^p_{t,x}(\R^{1+d})}
\end{equation}
holds for all $f \in L^p_{t,x}(\R^{1+d})$.  The constant $C$ depends only on $p$, $p_0$, $d$, the degree of $P$, and the constants in \eqref{E:elliptic restriction}.  
\end{theorem}

In particular, the restriction estimate \eqref{E:restriction} holds in the bilinear range $p < \tfrac{2(d+3)}{d+5}$.  As pointed out in \cite[Section~5.2]{LeeRogersSeeger}, the recent Bourgain--Guth \cite{BourgainGuth} and Guth \cite{Guth} theorems and the bilinear-to-linear method of \cite{TVV} establish $\scriptR(p \to q)$ (and hence Theorem~\ref{T:main}) in a slightly better (but awkward-to-state) range.  (This argument gives us an even better range in the monomial case because the existence of an almost transitive group action allows the use of the Maurey--Nikishin--Pisier factorization theorem \cite{Pisier} as in \cite{BourgainBesicovitch}.)  For $d \geq 3$, analogues of Theorem~\ref{T:main} were previously known only in the Stein--Tomas range \cite{CKZ1, CKZ2} (these results cover somewhat more general hypersurfaces).  

We will primarily focus on restriction with affine surface measure along the scaling line $q=\tfrac{dp'}{d+2}$ because this gives essentially the strongest possible estimates for such hypersurfaces.  However, in the last section, we will show how to deduce local (i.e.\ for compact pieces of the hypersurface) estimates from results off the scaling line $q=\tfrac{dp'}{d+2}$ (such as the Bourgain--Guth theorem \cite{BourgainGuth}), as well as sharp unweighted estimates.  

It should be possible to relax the hypotheses on $P$ substantially.  Evenness guarantees smoothness of $S_P$ and the vanishing of the linear term.  Neither smoothness at zero nor rotational symmetry are essential for our proof, and variants will be discussed in the last section.  The positivity of the coefficients and vanishing of the linear term, however, reflect geometric considerations that do play an important role.  Most obviously, the hypothesis that the coefficients are nonnegative rules out negatively curved hypersurfaces, for which no sharp restriction estimates are known beyond the Stein--Tomas range \cite{LeeNeg, VargasNeg}.  More subtly, since the linear term vanishes and the coefficients are positive, we can rescale dyadic annuli in $S_P$ to uniformly elliptic hypersurfaces.  That being said, in the last section, we will give a global, but non-uniform result for polynomials $P$ with $P''(t) > 0$ for all $t > 0$.

\subsection*{Sketch of proof} 
By duality, $\scriptR(p \to q)$ is equivalent to the adjoint restriction conjecture, which we denote by $\scriptR^*(q' \to p')$.  The adjoint restriction operator is also known as the extension operator, and we will say that an exponent pair $(p,q)$ is (extension) admissible if $(q',p')$ is restriction admissible, i.e.\ if 
$$
\tfrac{2(d+1)}d < q \leq \infty, \qquad q=\tfrac{(d+2)p'}d.
$$
It will generally be clear from the context whether an `admissible' pair is restriction or extension admissible.  

Our goal is to prove that for any admissible $(p_0,q_0)$, $\scriptR^*(p_0 \to q_0)$ implies that the extension operator
$$
\scriptE_P f(t,x) = \int e^{i(tP(|\xi|) + x\xi)}f(\xi)\, d\xi
$$
satisfies 
$$
\|\Lambda_P(|\nabla|)^{1/p'} \scriptE_Pf\|_{L^q_{t,x}} \lesssim \|f\|_{L^p_\xi}, \qquad f \in \scriptS(\R^d),
$$
for all admissible $(p,q)$ with $p < p_0$, with implicit constants depending on $d$, $p$, and the degree of $P$.  

We will proceed along the following lines.  Given a polynomial $P(t) = a_1t^2+\cdots + a_{N}t^{2N}$ with the $a_i$ non-negative, we may decompose $\R$ as a union of intervals, $\R = \bigcup_{j=1}^{C_N} I_j$, such that on $I_j$, $P$ behaves like the monomial $a_{j}t^{2j}$, plus a controllable error.  By the triangle inequality, it suffices to prove a uniform restriction estimate for each annular hypersurface $\{(P(|\xi|,\xi)) : |\xi| \in I_j\}$.  By affine invariance of \eqref{E:restriction}, we may assume that $a_{j}=1$.  The essential difficulty is then encapsulated by the problem of proving restriction estimates for degenerate hypersurfaces of the form $\{(|\xi|^{2j},\xi)\}$, for $j > 1$.  

By rescaling, the restriction problem on $\{(|\xi|^{2j},\xi) : |\xi| \sim 2^k\}$ is equivalent to restriction to $\{(|\xi|^{2j},\xi) : |\xi| \sim 1\}$.  The latter is (after a partition of unity) elliptic, so we can apply our hypothesis $\mathcal R(p_0 \to q_0)$, which implies $\mathcal R(p \to q)$ by interpolation.  This leaves us to control the interaction between the dyadic annuli and then sum up the dyadic pieces.  The former we do by means of a bilinear restriction estimate for transverse hypersurfaces whose curvatures are at different scales; after that the summation is almost elementary.

\subsection*{Prior results}  
As mentioned earlier, the natural conjectural form of Theorem~\ref{T:main} is for arbitrary polynomial hypersurfaces.  This is known if $d=2$ \cite{Sjolin}.  In fact, a uniform restriction result is known for polynomial curves with affine arclength measure in all dimensions (\cite{BAJM} and the references therein).

For hypersurfaces of dimension two or more, matters seem significantly more complicated.   Carbery--Kenig--Ziesler have proved uniform restriction theorems with affine surface measure in $1+2$ dimensions for $q \leq 2$ for rotationally symmetric hypersurfaces satisfying rather weak conditions on their derivatives \cite{CKZ2} (cf.\ \cite{OberlinPAMS04}) and for arbitrary homogeneous polynomials \cite{CKZ1}.  Ikromov--M\"uller \cite{IM} have proved the sharp unweighted $L^2_\xi$ restriction estimates for hypersurfaces in $\R^3$ expressed in adapted coordinates. 

Beyond the Stein--Tomas range, very little is known about restriction to degenerate hypersurfaces.  Lee--Vargas \cite{LeeVargas} have obtained restriction estimates in the bilinear range for hypersurfaces with $k$-nonvanishing principal curvatures; this result is in a somewhat different vein because the order of vanishing in the other directions is not taken into account in \cite{LeeVargas}.  In very recent independent work of Buschenhenke--M\"uller--Vargas (which has appeared in the Ph.D.\ thesis of Buschenhenke, \cite{Buschenhenke}; a version will be submitted for publication as \cite{BMV}), the authors establish a Fourier restriction theorem for convex finite type surfaces in $\R^3$ of the form $\{(\phi_1(\xi_1)+\phi_2(\xi_2),\xi):|\xi| \leq C\}$.  Both the form of the result and the methods are different (though there are some coincidental similarities in the proofs of the bilinear results).  In particular, the authors use the measure $d\xi$ (rather than the affine surface measure) and directly prove the corresponding scaling critical estimates, which necessarily depend on the $\phi_j$, in the bilinear range, without the use of the square function.  

\subsection*{Notation}  For two nonnegative quantities $A$ and $B$, the notation $A \lesssim B$ will be used to mean $A \leq CB$ for some constant $C$ that depends only on the dimension, degree of $P$ (or on the ellipticity parameters for more general results), and exponents $p,q,p_0$, unless otherwise stated.  We will write $A \sim B$ to mean $A \lesssim B$ and $B \lesssim A$, and $A = O(B)$ to mean $|A| \lesssim |B|$.  We will define the notation $A \lessapprox B$ later on (at the beginning of the proof of Lemma~\ref{L:epsilon removal} and at the end of Section~\ref{S:induction}) since its meaning will change.  The spatial Fourier transform, which acts on functions on $\R^d$, will be denoted by $f \mapsto \hat f$ and its inverse by $g \mapsto \check g$.  The spacetime Fourier transform, which acts on functions on $\R^{1+d}$, is denoted by $\scriptF$.  To simplify exponents, we will consistently ignore the fact that $2\pi \neq 1$.  

\subsection*{Acknowledgements} This work was supported by NSF grant DMS-1266336.  The author would like to thank Shuanglin Shao, Keith Rogers, Stefan Buschenhenke, Detlef M\"uller, and Ana Vargas for enlightening conversations along the way.  She would also like to thank the anonymous referee for helpful comments on the exposition.  

\section{Bilinear restriction I:  Statement of result}

We state our bilinear restriction result in the $C^\infty$, rather than polynomial, setting.  Let $c_0 >0$ and let $N$ be sufficiently small and large, respectively, dimensional constants.  Let $1 > \epsilon_0 > 0$, let $A < \epsilon_0$, and let $g_1,g_2 \in C^\infty(B(0,c_0))$ be elliptic phases (as defined in the previous section), which also satisfy the transversality condition $|\nabla g_1(0)| \lesssim |\nabla g_2(0)| \sim 1$; thus $|\nabla g_1| \lesssim|\nabla g_2|$ throughout $B(0,c_0)$.  The reader may find it helpful to keep the model case $g_1(\xi) = |\xi|^2$, $g_2(\xi) = |\xi-e_1|^2$ in mind.

Fix $J > 2$ and a pair of integers $k_1 > k_2$.  Define phase functions
$$
h_j(\xi) := 2^{-Jk_j}g_j(2^{k_j}\xi), \qquad \xi \in B(0,c_02^{-k_j}),
$$
surfaces
$$
S_j := \{(h_j(\xi),\xi) : \xi \in B(0,c_0 2^{-k_j})\},
$$
and extension operators
$$
\scriptE_j f(t,x) := \int_{\{|\xi| < c_0 2^{-k_j}\}} e^{i(t,x) \cdot(h_j(\xi),\xi)}f(\xi) \, d\xi, \qquad j=1,2.
$$

For simplicity, we state our bilinear result when $k_2 = 0$; the general case may be obtained by scaling.
\begin{theorem}\label{T:bilinear}
For $C = C_{d,\epsilon_0,J}$ sufficiently large, and all $k_1 \geq C$, $k_2=0$, $\delta > 0$, and $2 \geq q > \frac{d+3}{d+1}$,
\begin{equation} \label{E:bilinear scale 1}
\|\mathcal E_1 f_1 \mathcal E_2 f_2\|_{L^q_{t,x}} \lesssim_{\delta,q} 2^{k_1(J-2)(\frac1q-\frac12+\delta)}\|f_1\|_{L^2_\xi}\|f_2\|_{L_\xi^2}, \qquad f_1,f_2 \in L_\xi^2.  
\end{equation}
The implicit constant is allowed to depend on $\delta,q$, as well as $d,A,\epsilon_0,J$, but not on the phases $g_1,g_2$.  
\end{theorem}

\textit{Remarks:}  For $2 \leq q \leq \infty$, it is easy to prove this result without the exponential term; combining this with the theorem, we obtain the full range of estimates of the form $L^2 \times L^2 \to L^q$, excepting possibly the endpoint $q=\tfrac{d+3}{d+1}$.  We have not explored the optimal power of $2^{k_1}$ in \eqref{E:bilinear scale 1}.  In fact, the power given here is certainly not optimal since we do not use the small size of $S_1$.  On the other hand, our argument also works, with some modifications, when $h_1$ is replaced by $2^{-(J-2)k_1}g_1$.  

In 1+2 dimensions, bilinear adjoint restriction results have been proved in much greater generality by Buschenhenke--M\"uller--Vargas in \cite{BMV}.  It is also the author's understanding that they have independently obtained the above high-dimensional result using their methods (personal communication).  

The following scaling critical bilinear restriction result will be used in the proof of the linear restriction theorem.  

\begin{corollary}\label{C:bilin Str gen}
Assume that $\mathcal R^*(p_0 \to q_0)$ holds for some admissible pair $(p_0,q_0)$ with $p_0 > 2$, and assume that $N$ is large enough to satisfy both the hypotheses of the elliptic restriction theorem $\mathcal R^*(p_0 \to q_0)$ and of Theorem~\ref{T:bilinear}.  Then for all admissible pairs $(p,q)$ with $2 < p < p_0$, and any integers $k_1,k_2$, we have the bilinear extension estimate 
\begin{equation} \label{E:bilin Str gen}
\|(2^{-k_1 \frac{J-2}{q}} \mathcal E_1 f_1)( 2^{-k_2 \frac{J-2}{q}} \mathcal E_2 f_2)\|_{L^{\frac{q}2}_{t,x}} \lesssim  2^{-\delta_p |k_1-k_2|}\|f_1\|_{L^{p}_\xi}\|f_2\|_{L^{p}_\xi},
\end{equation}
for some $\delta_p>0$ depending only on $p,J$.  The implicit constant depends on $p,A,\eps_0,J$.  
\end{corollary}

\begin{proof}[Proof of Corollary~\ref{C:bilin Str gen}]
By considering the special case $q=\frac{d+2}{d}$ of the bilinear theorem and rescaling, we obtain the bilinear Stein--Tomas inequality
\begin{equation} \label{E:bilin Str}
2^{-\frac{(J-2)d}{2(d+2)}(k_1+k_2)}\|\mathcal E_1 f_1 \mathcal E_2 f_2\|_{L^{\frac{d+2}d}_{t,x}} \lesssim  2^{-c_d |k_1-k_2|} \|f_1\|_{L^2_\xi}\|f_2\|_{L^2_\xi},
\end{equation}
for some $c_d > 0$.  

Supposing that $\mathcal R^*(p_0 \to q_0)$ holds for some admissible pair, by rescaling we see that 
$$
\|2^{-k_j \frac{J-2}{q_0}} \mathcal E_j f\|_{L^{q_0}_{t,x}} \lesssim \|f\|_{L^{p_0}_\xi}, \qquad j=1,2.
$$
Thus by Cauchy--Schwarz,
$$
\|2^{-k_1 \frac{J-2}{q_0}} \mathcal E_1 f_1 2^{-k_2 \frac{J-2}{q_0}} \mathcal E_2 f_2\|_{L^{\frac{q_0}2}_{t,x}} \lesssim \|f_1\|_{L^{p_0}_\xi}\|f_2\|_{L^{p_0}_\xi},
$$
for any pair $k_1,k_2$.  By interpolation with \eqref{E:bilin Str}, we obtain the corollary.
\end{proof}

\textit{Remark:}  In any dimension, an $L^2_\xi \times L^2_\xi \to L^2_{t,x}$ estimate is easily proved by a well known argument using Plancherel, a change of variables, transversality (not curvature) of the hypersurfaces, and the support sizes.  In dimension 1+2, this yields  the improved bilinear Stein--Tomas estimate \eqref{E:bilin Str} directly, giving the corollary without the need for the bilinear machinery.  In higher dimensions, this does not quite work; we would want
$$
\begin{aligned}
&|\iint_{\{|\zeta_j| < c_0 2^{-k_j}\}} f_1(\zeta_1) f_2(\zeta_2) f_3(h_1(\zeta_1)+h_2(\zeta_2),\zeta_1+\zeta_2)\, d\zeta_1\, d\zeta_2| \\
&\qquad \qquad\qquad\qquad \qquad \lesssim 2^{(k_1+k_2)c_d(J-2)}\|f_1\|_{L^\frac{4d}{3d-2}_\xi}\|f_2\|_{L^\frac{4d}{3d-2}_\xi} \|f_3\|_{L^2_\xi},
\end{aligned}
$$
for $c_d > 0$ sufficiently small.  Unlike the $d=2$ case, however, the corresponding estimate for flat but transverse hypersurfaces is false, so curvature must play some role.  (The full range of exponents in the flat case is given in \cite{BCCTlong, BCCTshort}.)

\subsection*{Notation}  We use $\mathcal R^*(p \times p \to q)$ as shorthand for the statement that inequality \eqref{E:bilin Str gen} holds for extension operators $\mathcal E_1,\mathcal E_2$ as described in this section.

\section{Proof of the linear result} \label{S:linear}

This section will be devoted to a proof of Theorem~\ref{T:main}, using Corollary~\ref{C:bilin Str gen} from the previous section.  

For the remainder of the section, we assume that the adjoint restriction conjecture $\mathcal R^*(p_0 \to q_0)$ holds for some (extension) admissible pair $(p_0,q_0)$.  We may assume that $p_0 > 2$.  

Fix an admissible pair $(p,q)$ with $p < p_0$.  By interpolation with the trivial $L^1 \to L^\infty$ bound, $\mathcal R^*(p_0 \to q_0)$ implies that $\scriptR^*(p \to q)$ holds for all admissible pairs $(p,q)$ with $p \leq p_0$.  

Write $P(t) = a_0 + a_1 t^2 + \cdots + a_N t^{2N}$, with the $a_i$ nonnegative.  We may assume that $a_0=0$.  By duality, it suffices to prove that 
\begin{equation} \label{E:extension}
\|\Lambda_P(\nabla)^{\frac1{p'}} \mathcal E_P f\|_{L^q_{t,x}} \lesssim \|f\|_{L^p_\xi}
\end{equation}
for all $f \in L^p_\xi$ and admissible $(p,q)$ with $p < p_0$, where the implicit constant depends on $p,N$.  Here $\Lambda_P(\nabla)$ denotes the Fourier multiplication operator with symbol $\Lambda_P(\xi)$.  

\subsection{Initial decomposition}

We begin by decomposing $(0,\infty)$ as a union of intervals on which $P$ is essentially monomial-like.  

Define 
$$
J_j := \{t \in (0,\infty) : a_j t^{2j} = \max_{1 \leq i \leq N} a_i t^{2i}\}.
$$
Then the $J_j$ are consecutive intervals, intersecting only at their boundaries, and $(0,\infty) = \bigcup_{j=1}^N J_j$.  By the triangle inequality, it suffices to prove \eqref{E:extension} for $f$ supported on a single annulus $\{|\xi| \in J_j\}$.  The low frequency case is easy.

\begin{lemma} \label{L:low freq}
Let $B_1 = \{0\} \cup \{|\xi| \in J_1\}$.  We have the estimate
\begin{equation} \label{E:low freq}
\|\Lambda_P(\nabla)^{\frac1{p'}} \mathcal E_P \chi_{B_1} f\|_{L^q_{t,x}} \lesssim \|f\|_{L^p_\xi},
\end{equation}
with uniform implicit constants.  
\end{lemma}

\begin{proof} By rescaling, we may assume that $B_1$ equals $B$, the unit ball.  By applying an affine transformation, we may assume that $a_1 = 1$.  Then by the definition of $J_1$, $a_j \leq 1$, $2 \leq j \leq N$, so $\Lambda_P \sim 1$ on $B$ and \eqref{E:low freq} just follows from our assumption that $\mathcal R^*(p_0 \to q_0)$ (and hence $\mathcal R^*(p \to q)$) holds.  
\end{proof}

\subsection{Dyadic decomposition}
Fix an integer $j \geq 2$.  By applying an affine transformation, we may assume that $a_j = 1$.  

Let $I_k := J_j \cap [2^{-k-1},2^{-k}]$.  Assume that $I_k \neq \emptyset$.  We will assume that $I_k = [2^{-k-1},2^{-k}]$.  (For simplicity we ignore intervals containing the endpoints of the $J_j$; they may be treated similarly, and there are only a bounded number of them anyway.)  Let $A_k:= \{\xi:|\xi| \in I_k\}$.  Consider the phase 
$$
g_k(\xi) := 2^{2jk}g(2^{-k}|\xi|) = |\xi|^{2j} + \sum_{i \neq j} a_i 2^{2(j-i)k}|\xi|^i, \qquad \xi \in A_0.
$$
Since $2^{-k}\in J_j$, we have $2^{2(j-i)k}a_i \leq 1$, $i \neq j$, so $g_k$ is elliptic (with the parameters $A,\eps$ depending only on the degree of $P$).  Thus by our hypothesis that $\mathcal R^*(p \to q)$ holds for admissible $(p,q)$ with $p \leq p_0$, and rescaling, we have the following.

\begin{lemma}\label{L:dyadic freq}
For any $k \in \Z$,
$$
\|\Lambda_P(\nabla)^{\frac1{p'}}\mathcal E_P \chi_{A_k} f\|_{L^q_{t,x}} \lesssim \|f\|_{L^p_\xi}.
$$
\end{lemma}

\subsection{Almost orthogonality.}  The next lemma establishes a decay estimate for the interaction between annular pieces at different scales.  

\begin{lemma}  \label{L:dyadic bilin}
For any integers $k_1,k_2$ such that $I_{k_i} \cap J_j \neq \emptyset$ for $i=1,2$ and some $j \geq 2$,
\begin{equation} \label{E:dyadic bilin}
\bigl\|\bigl(\Lambda_P(\nabla)^{\frac1{p'}} \mathcal E_P \chi_{A_{k_1}} f_1\bigr)\bigl( \Lambda_P(\nabla)^{\frac1{p'}} \mathcal E_P \chi_{A_{k_2}} f_2\bigr)\bigr\|_{L^{\frac{q}2}_{t,x}} \lesssim 2^{-\delta|k_1-k_2|}\|f_1\|_{L^p_\xi}\|f_2\|_{L^p_\xi}.
\end{equation}
\end{lemma}

\begin{proof}
We know that the $g_k$ are elliptic with uniform parameters; let these be denoted by $A,\eps$.  Since $|\nabla g_k| \sim 1$ on $A_0$, we may decompose $A_0$ as a finite union of balls of radius $c_0$, with $c_0$ sufficiently small that $A c_0 \ll \eps$ (as was required for Theorem~\ref{T:bilinear}).  Then \eqref{E:dyadic bilin} follows from Corollary~\ref{C:bilin Str gen} and the triangle inequality.
\end{proof}

\subsection*{Summation} Now we put the pieces together.  Using boundedness of the Littlewood--Paley square function, Minkowski's inequality and the fact that $q \leq 4$, Lemma~\ref{L:dyadic bilin}, the fact that $\delta > 0$, and finally the fact that $q > p$, we have for any $j \geq 2$ that
\begin{align*}
&\|\Lambda_P(\nabla)^{\frac1{p'}} \mathcal E_P \chi_{\{|\xi| \in J_j\}} f\|_{L^q_{t,x}}^q
\lesssim \int \left(\sum_k |\Lambda_P(\nabla)^{\frac1{p'}}\mathcal E_P \chi_{A_k} f|^2\right)^{\frac q2}\, dx\, dt\\
&\qquad \lesssim \sum_{k_1 \leq k_2} \int |(\Lambda_P(\nabla)^{\frac1{p'}} \mathcal E_P \chi_{A_{k_1}} f)(\Lambda_P(\nabla)^{\frac1{p'}} \mathcal E_P \chi_{A_{k_2}} f)|^{\frac q2}\, dx\, dt\\
&\qquad \lesssim  \sum_{k_1 \leq k_2} 2^{-\frac q2 \delta|k_1-k_2|} \|\chi_{A_{k_1}} f\|_{L^p_\xi}^{\frac q2}\|\chi_{A_{k_2}}f\|_{L^p_\xi}^{\frac q2} \lesssim \sum_k \|\chi_{A_k} f\|_{L^p_\xi}^q \leq \|f\|_{L^p_\xi}^q.  
\end{align*}
This completes the proof of Theorem~\ref{T:main}, modulo the proof of Theorem~\ref{T:bilinear}.  \qed  

We note that related applications of square functions (albeit more complex ones) have also appeared in the work \cite{CKZ1, CKZ2} of Carbery--Kenig--Ziesler.  

We will give the proof of Theorem~\ref{T:bilinear} over the next 6 sections.  The argument is essentially that of Tao in \cite{TaoParab}, but modifications are needed throughout to deal with the degenerate curvature.  

\section{Preliminary reductions}

After making an invertible affine transformation of the frequency space $\R^{1+d}$, we may assume that $\nabla g_1(0) = 0$, that $D^2 g_1(0)=I_d$ (the identity), that $\nabla g_2(0) = e_1$ (the first coordinate vector), and that $D^2 g_2(0)$ is positive definite with eigenvalues comparable to 1.  We recall that the hypersurface $S_1$ is at scale $2^{-k_1}$, with $k_1 \geq C_{d,J,\eps} \gg 1$, while $S_2$ is at scale 1.  

For $R \geq 1$, let $Q_R$ denote the set
$$
Q_R = \{(t,x) \in \R^{1+d} : \tfrac12 2^{k_1(J-2)}R \leq t \leq 2^{k_1(J-2)}R, \quad |x| \leq R\}.
$$
The main step in the proof of our bilinear restriction theorem is the following local estimate.  

\begin{proposition} \label{P:local rest}
For every $\delta > 0$ and $R \geq 1$, 
\begin{equation} \label{E:local rest}
\|\mathcal E_1 f_1 \mathcal E_2 f_2\|_{L^{\frac{d+3}{d+1}}_{t,x}(Q_R)} \lesssim_\delta 2^{k_1\delta} R^\delta 2^{\frac{k_1(J-2)(d-1)}{2(d+3)}}\|f_1\|_{L^2_\xi}\|f_2\|_{L^2_\xi}, \qquad f_2,f_2 \in L^2.
\end{equation}
The implicit constant is allowed to depend on $d, J, \delta$, but not on $R$ or $k_1$.  
\end{proposition}

The remainder of this section will be devoted to a proof of the sufficiency of Proposition~\ref{P:local rest}.  By interpolation with the easy estimate
\begin{equation} \label{E:easy L2}
\|\mathcal E_1 f_1 \mathcal E_2 f_2\|_{L^2_{t,x}} \lesssim \|f_1\|_{L^2_\xi}\|f_2\|_{L^2_\xi},
\end{equation}
it suffices to prove the following ``epsilon removal'' lemma.

\begin{lemma} \label{L:epsilon removal}
Assuming Proposition~\ref{P:local rest}, for any $\delta > 0$ and $\tfrac{d+2}d > q > \tfrac{d+3}{d+1}$, 
\begin{equation} \label{E:epsilon removal}
\|\mathcal E_1 f_1 \mathcal E_2 f_2\|_{L^q_{t,x}} \lesssim_{\delta,q} 2^{k_1 \delta} 2^{\frac{k_1(J-2)(d-1)}{2(d+3)}}\|f_1\|_{L^2_\xi}\|f_2\|_{L^2_\xi}.
\end{equation}
\end{lemma}

\begin{proof}
The basic argument is essentially that of \cite{BourgainCone, TV1}, but adjustments are needed throughout to account for the degeneracy of $S_2$ and to obtain the precise power in \eqref{E:epsilon removal}.  For the convenience of the reader, we give the brief proof.  

For the remainder of the section, we will use the notation $A \lessapprox B$ if $A \lesssim_\delta 2^{k_1\delta} B$ for each $\delta > 0$.  

Fix a nonnegative $\phi \in C^\infty_{c}(\R^{1+d})$ with $\phi \equiv 1$ on $\{|(t,x)| \leq 1\}$ and $\sum_{m \in \Z^{d+1}} \phi(\cdot-m) \sim 1$.

We will actually prove that if the local estimate \eqref{E:local rest} holds for slightly expanded surfaces,
$$
S_j = \{(h_j(\xi),\xi) : |\xi| < 6 c_0  2^{-k_j}\},
$$
and corresponding $\mathcal E_j$, then the bilinear restriction estimate $R^*(2\times2 \to q)$ holds for all $q > \frac{d+3}{d+1}$, but for simplicity we will gloss over the fact that $6 \neq 1$ by using the same notation for these expanded $S_j,\mathcal E_j$.

By interpolation with the $L^2$ estimate, it suffices to prove the weak type estimates
$$
|\{|\mathcal E_1 f_1 \scriptE_2 f_2| > \lambda\}| \lessapprox 2^{\frac{k_1(J-2)(d-1)}{2(d+3)}} \|f_1\|_{L^2_\xi} \|f_2\|_{L^2_\xi} \lambda^{-q}, \quad \tfrac{d+2}d > q > \tfrac{d+3}{d+1}.
$$
This in turn may be reduced to proving that
\begin{equation} \label{E:EEE_1}
\|\chi_E \mathcal E_1 f_1 \scriptE_2 f_2\|_{L^1_{t,x}} \lessapprox 2^{\frac{k_1(J-2)(d-1)}{2(d+3)}} |E|^{\frac1{q'}} \|f_1\|_{L^2_\xi}\|f_2\|_{L^2_\xi},
\end{equation}
for all Borel sets $E$.

Fix $f_1$.  By duality \eqref{E:EEE_1} would follow from
$$
\|\scriptE_2^*(\chi_E \scriptE_1 f_1 F_2)\|_{L^2_\xi} \lessapprox 2^{\frac{k_1(J-2)(d-1)}{2(d+3)}} |E|^{\frac1{q'}}\|f_1\|_{L^2_\xi} \|F_2\|_{L^\infty_{t,x}}.
$$
By Plancherel,
$$
\|\scriptE_2^*(\chi_E \scriptE_1 f_1 F_2)\|_{L^2_\xi}^2 = \langle (\chi_E\scriptE_1 f_1 F_2)*\scriptE_21,\chi_E\scriptE_1 f_1F_2\rangle.
$$
Let 
\begin{equation} \label{E:def R2}
R_2 := \max\{1,2^{k_1(J-2)\frac2d(\frac{d}{d+2}-\frac{d-1}{d+3})}|E|^{\frac4d(\frac{d+4}{2(d+2)}-\frac1{q'})}\},
\end{equation}
and define functions
\begin{gather*}
\phi_{2,R_2}(t,x) = \phi(\tfrac{t}{R_2}, \tfrac{x}{R_2}), \qquad \phi_{2,R_2}^c = 1-\phi_{2,R_2}\\
\psi_{2,R_2} = \phi_{2,R_2} \scriptE_2 1, \qquad \psi_{2,R_2}^c = \phi_{2,R_2}^c \scriptE_2 1.
\end{gather*}

By stationary phase,
$$
\|\psi_{2,R_2}^c\|_{L^\infty_{t,x}} \lesssim R_2^{-\frac d2}.
$$
By H\"older and Stein--Tomas (rescaled),
$$
\|\chi_E \scriptE_1 f_1 F_2\|_{L^1_{t,x}} \lesssim 2^{\frac{k_1(J-2)d}{2(d+2)}}|E|^{\frac{d+4}{2(d+2)}}\|f_1\|_{L^2_\xi}\|F_2\|_{L^\infty_{t,x}}.
$$
Hence
$$
\langle (\chi_E \scriptE_1 f_1 F_2)*\psi_{2,R_2}^c,\chi_E\scriptE_1 f_1 F_2\rangle \lesssim 2^{\frac{k_1(J-2)d}{d+2}}R_2^{-\frac d2}|E_2|^{\frac{d+4}{d+2}} \|f_1\|_{L^2_\xi}^2 \|F_2\|_{L^\infty_{t,x}}^2.
$$
Using \eqref{E:def R2}, we see that this is acceptable, so we turn to the main term.

Let $\mu_2$ denote surface measure on $S_2$; then (using rapid decay of $\psi_{2,R_2}$),
$$
\scriptF(\psi_{2,R_2}) = \scriptF(\phi_{2,R_2})*\mu_2 \lesssim \sum_{j=0}^\infty 2^{-Mj} R_2 \chi_{S_{2,2^j R^{-1}}}, 
$$
where $M$ is sufficiently large for later purposes and
$$
S_{2,2^jR_2^{-1}} := \begin{cases} \{(\tau,\xi) : |\xi| < 2c_0, \: |\tau-h_2(\xi)| < 2^j R_2^{-1}\}, \quad &2^j \ll R_2 \\
\{(\tau,\xi) : |(\tau,\xi)| < 2^j R_2^{-1}\}, \quad & 2^j \gtrsim R_2.\end{cases}
$$
Using this and Plancherel,
\begin{align*}
\langle (\chi_E \scriptE_1 f_1 F_2) * \psi_{2,R_2},\chi_E \scriptE_1 f_1 F_2 \rangle &= \| \scriptF(\chi_E \scriptE_1 f_1 F_2)\|_{L^2_{\tau,\xi}(\scriptF(\psi_{2,R_2}))}^2 \\
&\lesssim \sum_{j=0}^\infty 2^{-Mj} R_2 \|\scriptF(\chi_E \scriptE_1 f_1 F_2)\|_{L^2_{\tau,\xi}(S_{2,2^j R_2^{-1}})}^2.
\end{align*}
By a simple covering argument (and translation invariance of our inequality), it suffices to consider the $j=0$ case.  We want
$$
\|\scriptF(\chi_E \scriptE_1 f_1 F_2)\|_{L^2_{\tau,\xi}(S_{2,R_2^{-1}})} \lessapprox 2^{\frac{k_1(J-2)(d-1)}{2(d+3)}}R_2^{-\frac12} |E|^{\frac1{q'}} \|f_1\|_{L^2_\xi}\|F_2\|_{L^\infty_{t,x}}.
$$
By Plancherel and duality, this is equivalent to
\begin{equation} \label{E:EEEtilde}
\|\chi_E \scriptE_1 f_1 \tilde\scriptE_2\tilde f_2\|_{L^1_{t,x}} \lessapprox 2^{\frac{k_1(J-2)(d-1)}{2(d+3)}}R_2^{-\frac12}|E|^{\frac1{q'}}\|f_1\|_{L^2_\xi}\|\tilde f_2\|_{L^2_{\tau,\xi}},
\end{equation}
where 
$$
\tilde \scriptE_2 \tilde f_2 = \scriptF^*(\chi_{S_{2,R_2}}\tilde f_2).
$$

Now fix $\tilde f_2 \in L^2$.  By duality and Plancherel, \eqref{E:EEEtilde} is equivalent to 
$$
\langle(\chi_E F_1 \tilde\scriptE_2 \tilde f_2)*\scriptE_1 1, \chi_E F_1 \tilde \scriptE_2 \tilde f_2\rangle \lessapprox 2^{\frac{k_1(J-2)(d-1)}{d+3}} R_2^{-1} |E|^{\frac2{q'}}\|F_1\|_{L^\infty_{t,x}}^2 \|\tilde f_2\|_{L^2_{\tau,\xi}}^2.
$$
Let
$$
R_1 := \max\{1, 2^{-\frac{2k_1(J-2)(d-1)}{d(d+3)}} |E|^{\frac4d(\frac{d+4}{2(d+2)}-\frac1{q'})}\},
$$
and define
\begin{gather*}
\phi_{1,R_1}(t,x) = \phi(\tfrac{t}{2^{k_1(J-2)}R_1},\tfrac x{R_1}), \qquad \phi_{1,R_1}^c = 1-\phi_{1,R_1}\\
\psi_{1,R_1} = \phi_{1,R_1} \scriptE_1 1, \qquad \psi_{1,R_1}^c = \phi_{1,R_1}^c \scriptE_1 1.
\end{gather*}
Using stationary phase\footnote{In fact, a better stationary phase estimate is possible, but we use the one that also works when $h_1$ is replaced by $2^{-(J-2)k_1}g_1$; similarly for \eqref{E:Fpsi1}.}, H\"older, and Stein--Tomas as before,
$$
\langle (\chi_E F_1 \tilde \scriptE_2 \tilde f_2)*\psi_{1,R_1}^c, \chi_E F_1 \tilde \scriptE_2 \tilde f_2\rangle \lesssim R_1^{-\frac d2} R_2^{-1}|E|^{\frac{d+4}{d+2}} \|F_1\|_{L_{t,x}^\infty}^2 \|\tilde f_2\|_{L_\xi^2}^2,
$$
which is acceptable.  

We compute
\begin{equation} \label{E:Fpsi1}
\scriptF(\psi_{1,R_1}) = \scriptF(\phi_{1,R_1})*\mu_1 \lesssim 2^{k_1(J-2)}R_1 \sum_{j=0}^\infty 2^{-Mj} \chi_{S_1,2^{jR^{-1}}}, 
\end{equation}
where
$$
S_{1,2^jR^{-1}} := \begin{cases} \{(\tau,\xi) : |\xi| < 2 c_0 2^{-k_1}, |\tau-h_1(\xi)| < 2^j 2^{-k_1(J-2)} R_1^{-1}\}, \quad & 2^j \ll R\\
\{(\tau,\xi) : |\xi| < 2^j R^{-1}, |\tau| < 2^j 2^{-k_1(J-2)}R_1^{-1}\}, \quad & 2^j \gtrsim R. \end{cases}
$$
Thus to estimate the main term, it suffices to show that
$$
\|\scriptF(\chi_E F_1 \tilde \scriptE_2 \tilde f_2 )\|_{L^2_{\tau,\xi}(S_{1,R_1^{-1}})} \lessapprox 2^{k_1(J-2)(-\frac12 + \frac{d-1}{2(d+3)})}(R_1R_2)^{-\frac12} |E|^{\frac1{q'}} \|F_1\|_{L_{t,x}^\infty} \|\tilde f_2\|_{L_{\tau,\xi}^2},
$$
or equivalently,
\begin{align*}
&\|\chi_E \tilde\scriptE_1\tilde f_1 \tilde\scriptE_2 \tilde f_2\|_{L_{t,x}^1} \\
&\qquad  \lessapprox 2^{k_1(J-2)(-\frac12 + \frac{d-1}{2(d+3)})}(R_1R_2)^{-\frac12}|E|^{\frac1{q'}} \|\tilde f_1\|_{L^2_{\tau,\xi}} \|\tilde f_2\|_{L^2_{\tau,\xi}}, \qquad \tfrac{d+2}d > q > \tfrac{d+3}{d+1},
\end{align*}
where 
$$
\tilde\scriptE_1 \tilde f_1 = \scriptF^*(\chi_{S_{1,R_1}} \tilde f_1).
$$
By H\"older and the definition of $R_1,R_2$, this would follow from
\begin{align} \label{E:EtildeEtilde}
&\|\tilde\scriptE_1 \tilde f_1 \tilde \scriptE_2\tilde f_2\|_{L_{t,x}^\frac{d+3}{d+1}} \\\notag
&\qquad \lesssim_{\delta} 2^{k_1 \delta} (R_1R_2)^{\delta} 2^{k_1(J-2)(-\frac12+\frac{d-1}{2(d+3)})}(R_1R_2)^{-\frac12}\|\tilde f_1\|_{L_{\tau,\xi}^2}\|\tilde f_2\|_{L_{\tau,\xi}^2}, \quad \delta > 0.
\end{align}

In proving \eqref{E:EtildeEtilde}, we may assume that $\supp \tilde f_j \subseteq S_{j,R_j}$, $j=1,2$.  To avoid a proliferation of tildes, we will let $\tilde\scriptE_j f := \scriptF^*(\chi_{S_{j,3R_j}} f)$.  Let $\varphi$ be a smooth non-negative function with $\sum_{m \in \Z^{d+1}} \varphi(\cdot-m) \sim 1$ and $\hat \varphi$ supported in $\{|(\tau,\xi)| \leq 1\}$.  For $(t_0,x_0) \in \R^{1+d}$, define
$$
\varphi^{(t_0,x_0)}_{R_1,R_2}(t,x) = \varphi(\tfrac{t-t_0}{2^{k_1(J-2)}R_1}, \tfrac{x-x_0}{R_2}), \qquad \phi^{(t_0,x_0)}_{R_1,R_2}(t,x) = \phi(\tfrac{t-t_0}{2^{k_1(J-2)}R_1},\tfrac{x-x_0}{R_2}).
$$
Then 
\begin{equation} \label{E:sum is 1}
\sum_{(t_0,x_0)} \phi^{(t_0,x_0)}_{R_1,R_2}(\varphi^{(t_0,x_0)}_{R_1,R_2})^2 \sim 1 \sim \sum_{(t_0,x_0)}\varphi^{(t_0,x_0)}_{R_1,R_2},
\end{equation}
where the sum is taken over $(t_0,x_0) \in (2^{k_1(J-2)}R_1\Z) \times (R_2 \Z)^d$.  

Using the triangle inequality and \eqref{E:sum is 1}, our assumptions on the supports of $\phi,\hat\varphi$, the local restriction estimate with Fubini, and finally Cauchy--Schwarz, Plancherel, and \eqref{E:sum is 1} again,
\begin{align*}
&\|\tilde\scriptE_1 \tilde f_1 \tilde\scriptE_2 \tilde f_2 \|_{L_{t,x}^{\frac{d+3}{d+1}}(\R^{1+d})} \lesssim \sum_{(t_0,x_0)} \|\phi^{(t_0,x_0)}_{R_1,R_2} (\varphi^{(t_0,x_0)}_{R_1,R_2} \tilde\scriptE_1 \tilde f_1)( \varphi^{(t_0,x_0)}_{R_1,R_2} \tilde\scriptE_2\tilde f_2)\|_{L_{t,x}^{\frac{d+3}{d+1}}(\R^{1+d})}\\
&\qquad \lesssim \sum_{(t_0,x_0)} \|\tilde \scriptE_1(\hat{\varphi^{(t_0,x_0)}_{R_1,R_2}} *\tilde f_1)\tilde \scriptE_2(\hat{\varphi^{(t_0,x_0)}_{R_1,R_2}} *\tilde f_2)\|_{L_{t,x}^{\frac{d+3}{d+1}}(Q_{6R_2}^{(t_0,x_0)})}\\
&\qquad \lesssim_{\delta} \sum_{(t_0,x_0)} 2^{k_1\delta}R_2^\delta (2^{k_1(J-2)}R_1R_2)^{-\frac12} 2^{\frac{k_1(J-2)(d-1)}{2(d+3)}}\\
&\qquad\qquad\qquad\qquad \times\|\scriptF(\varphi_{R_1,R_2}^{(t_0,x_0)})*\tilde f_1\|_{L_{\tau,\xi}^2}\|\scriptF(\varphi_{R_1,R_2}^{(t_0,x_0)})*\tilde f_2\|_{L_{\tau,\xi}^2}\\
&\qquad \lesssim_{\delta} 2^{k_1\delta}R_2^\delta (2^{k_1(J-2)}R_1R_2)^{-\frac12} 2^{\frac{k_1(J-2)(d-1)}{2(d+3)}} \|\tilde f_1\|_{L^2_{\tau,\xi}} \|\tilde f_2\|_{L^2_{\tau,\xi}},
\end{align*}
which is what we wanted.  This completes the proof.  
\end{proof}

\section{Induction} \label{S:induction}

Let $\mathcal R^*(2 \times 2 \to \tfrac{d+3}{d+1}; \delta, \alpha)$ denote the statement that the local estimate
\begin{equation}\label{E:R*loc}
\|\scriptE_1 f_1  \scriptE_2 f_2\|_{L^{\frac{d+3}{d+1}}_{t,x}(Q_R)} \lesssim_{\delta,\alpha} 2^{\delta k_1} R^\alpha 2^{\frac{k_1(J-2)(d-1)}{2(d+3)}} \|f_1\|_{L^2_\xi} \|f_2\|_{L^2_\xi}, 
\end{equation}
holds for all $R \geq 1$ and $f_1,f_2 \in L^2_\xi$.  

\begin{lemma}\label{L:base case}
For all $\delta>0$, $\scriptR(2\times2 \to \tfrac{d+3}{d+1};\delta,\tfrac{d^2-1}{2(d+3)}+\tfrac12)$ holds.
\end{lemma}

Assuming the lemma, Proposition~\ref{P:local rest} would follow from
\begin{equation} \label{E:induction step}
\scriptR^*(2 \times 2 \to \tfrac{d+3}{d+1}; \delta,\alpha) \implies \scriptR^*(2 \times 2 \to \tfrac{d+3}{d+1}; \delta+C\eps', \max\{(1-\eps)\alpha,C\eps\}+C\eps'\}),
\end{equation}
for all $\alpha > 0$ and $1 \gg \delta,\eps,\eps' > 0$.  We will prove \eqref{E:induction step} in Sections~\ref{S:packet}--\ref{S:combinat}, using Wolff's induction on scales argument from \cite{WolffCone}  (more precisely, a variant of Tao's adaptation in \cite{TaoParab}).  We turn now to the proof of Lemma~\ref{L:base case}.  

\begin{proof}[Proof of Lemma~\ref{L:base case}]
Let $\delta > 0$.  We may assume that $R \gtrsim 2^{\delta k_1(J-2)}$.  (For smaller $R$, use the bound for $\|\scriptE_1 f_1 \scriptE_2 f_2\|_{L^{\frac{d+3}{d+1}}_{t,x}(Q_{2^{\delta k_1(J-2)}})}$).  We also assume that $\|f_1\|_{L^2_\xi} = \|f_2\|_{L^2_\xi} = 1$.  

By H\"older's inequality,
\begin{align*}
\|\scriptE_1 f_1 \scriptE_2 f_2\|_{L^{\frac{d+3}{d+1}}_{t,x}(Q_R)} &\lesssim |Q_R|^{\frac{d+1}{d+3}-\frac12}\|\scriptE_1 f_1\|_{L^\infty_{t,x}} \|\scriptE_2 f_2\|_{L^2_{t,x}(Q_R)}\\
&\lesssim 2^{k_1(J-2)\frac{d-1}{2(d+3)}}R^{\frac{d^2-1}{2(d+3)}}\|f_1\|_{L^2_\xi}\|\scriptE_2 f_2\|_{L^2_{t,x}(Q_R)},
\end{align*}
so it suffices to show that 
\begin{equation} \label{E:E2f2L2}
\|\scriptE_2 f_2\|_{L^2_{t,x}(Q_R)} \lesssim_\delta R^{\frac12}.
\end{equation}

When $k_1 = 0$, \eqref{E:E2f2L2} just follows from H\"older's inequality (in the time direction) and Plancherel.  For larger $k_1$, $Q_R$ is tall and thin, so we decompose it as a union of cubes:  
$$
Q_R = \bigcup_{j=0}^{2^{k_1(J-2)}} Q_j',
$$
where 
$$
Q_j' = Q_j \cap Q_R, \:\ctc{and}\: \:Q_j = \{(t,x) : Rj \leq t \leq R(j+1), \: |x| \leq R\}.
$$
The idea of the proof of \eqref{E:E2f2L2} is that on $Q_j'$, $\scriptE_2 f_2$ is well-approximated by a function $f_2^{(j)}$ whose extension is spatially localized at time $Rj$.  Moreover, for $j \neq k$, these pieces are essentially orthogonal.  

To make this heuristic rigorous, fix a smooth, non-negative function $\phi$ with $\phi \equiv 1$ on $\{|\xi| < 2\}$ and $\phi \equiv 0$ off $\{|\xi| < 3\}$.  For $j \in \Z$, define
$$
f_2^{(j)}(\xi) := e^{-iRj h_2(\xi)}\phi(\tfrac\xi{c_0}) [\phi(\tfrac x{CR}) \scriptE_2 f_2(Rj,x)]\,\hat{\:}\,(\xi),
$$
where the inner Fourier transform is taken with respect to the $x$ variable.  

\begin{lemma} \label{L:f2j error}
For $(t,x) \in Q_j'$ and $M \geq 0$,
\begin{equation} \label{E:f2j error}
|\scriptE_2 f_2(t,x) - \scriptE_2 f_2^{(j)}(t,x)| \lesssim_M R^{-M}.
\end{equation}
\end{lemma}

\begin{lemma} \label{L:CotlarStein}
\begin{equation} \label{E:CotlarStein}
\sum_{j=0}^{2^{k_1(J-2)}} \|f_2^{(j)}\|_{L^2_\xi}^2 \lesssim \|f_2\|_{L^2_\xi}^2.
\end{equation}
\end{lemma}

We postpone the proofs of Lemmas~\ref{L:f2j error} and~\ref{L:CotlarStein} while we complete the proof of \eqref{E:E2f2L2}.  Choosing $M$ sufficiently large depending on $\delta$, and using \eqref{E:f2j error} together with H\"older's inequality, Plancherel, and finally \eqref{E:CotlarStein},
$$
\|\scriptE_2 f_2\|_{L^2_{t,x}(Q_R)} \lesssim 1+(\sum_{j=0}^{2^{k_1(J-2)}} \|\scriptE_2 f_2^{(j)}\|_{L^2_{t,x}(Q_j')}^2)^{\frac12} 
\lesssim 1+(\sum_{j=0}^{2^{k_1(J-2)}} R\|f_2^{(j)}\|_{L^2_\xi}^2)^{\frac12}
\lesssim R^{\frac12},
$$
and \eqref{E:E2f2L2} (and hence Lemma~\ref{L:base case}) is proved.
\end{proof}

\begin{proof}[Proof of Lemma~\ref{L:f2j error}]
Because $\supp f_2 \subseteq \{|\xi| < c_0\}$,
$$
\scriptE_2 f_2(t,x) = \iiint e^{i(t-Rj,x-y)(h_2(\eta),\eta)} \phi(\tfrac\eta{c_0}) e^{i(Rj,y)(h_2(\xi),\xi)} \phi(\tfrac\xi{c_0}) f_2(\xi)\, d\xi\, dy\, d\eta.
$$
Thus
\begin{equation} \label{E:E2-E2j}
\scriptE_2 f_2(t,x) - \scriptE_2 f_2^{(j)}(t,x) = \int P(t,x;\xi) e^{iRj h_2(\xi)} \phi(\tfrac\xi{c_0})f_2(\xi)\, d\xi,
\end{equation}
where
\begin{equation} \label{E:P}
P(t,x;\xi) = \iint e^{i(t-Rj,x-y)(h_2(\eta),\eta)}\phi(\tfrac\eta{c_0}) \, d\eta \, e^{iy\xi}(1-\phi(\tfrac y{CR}))\, dy.
\end{equation}

For $(t,x) \in Q_j'$ and $|y| > CR$, $|t-Rj| \leq R$ and $|x-y| \geq |y|-R$, so
$$
|\nabla_\eta(t-Rj,x-y)(h_2(\eta),\eta)| = |(t-Rj)\nabla h_2(\eta)+(x-y)| \gtrsim |y|,
$$
so integrating by parts in the inner integral of \eqref{E:P},
$$
|P(t,x;\xi)| \lesssim_M \int(1+|y|)^{-(M+d)}(1-\phi(\tfrac y{CR})) \, dy \lesssim R^{-M}.
$$
Inserting this in \eqref{E:E2-E2j} and using H\"older (and $\|f_2\|_{L^2_\xi} \sim 1$) gives
$$
|\scriptE_2 f_2(t,x) - \scriptE_2 f_2^{(j)}(t,x)| \lesssim_M R^{-M}\|f_2\|_{L^1_\xi} \lesssim R^{-M}.
$$
\end{proof}

\begin{proof}[Proof of Lemma~\ref{L:CotlarStein}]
Define
$$
T_j f(\xi) = e^{-iR j h_2(\xi)}\phi(\tfrac{\xi}{c_0})[\phi(\tfrac\cdot{CR})\scriptE_2 (\phi(\tfrac{\eta}{c_0}) f(\eta))(Rj,\cdot)]\hat{\:}(\xi).
$$
Then by the support condition on $f_2$, $f_2^{(j)} = T_j f_2$.  Each $T_j$ is self-adjoint.  When $|k-j| \gg 1$, we compute
$$
T_k^*T_j f(\xi) = e^{-iRkh_2(\xi)}\phi(\tfrac{\xi}{c_0}) \int K_{jk}(\xi,\zeta) e^{i R_j h_2(\zeta)} \phi(\tfrac{\zeta}{c_0})f(\zeta)\, d\zeta,
$$
where
$$
K_{jk}(\xi,\zeta) = \iint e^{i(y\zeta-x\xi)}\phi(\tfrac y{CR})\phi(\tfrac x{CR}) \int e^{i(Rk-Rj,x-y)(h_2(\eta),\eta)}\phi(\tfrac\eta{c_0})^2\, d\eta\, dx\, dy.
$$
On the support of the integrand, $|x-y| \lesssim R$ so for $|k-j| \gg 1$, 
$$
|\nabla_\eta [(Rk-Rj,x-y)\cdot(h_2(\eta),\eta)]| \gtrsim R|k-j|.
$$
  Integrating by parts $M+2d$ times in the inner integral and using H\"older's inequality,
$$
|K_{jk}(\xi,\zeta)| \lesssim_M R^{-M}|k-j|^{-M}.
$$
Applying H\"older's inequality again, we thus see that 
\begin{equation} \label{E:CShypo}
\|T_k^*T_j f\|_{L^2_\xi} = \|T_kT_j^* f\|_{L^2_\xi} \lesssim R^{-M}(1+|k-j|)^{-M}\|f\|_{L^2_\xi}, \qquad f \in L^2;
\end{equation}
by Plancherel, this is also valid for $|k-j| \lesssim 1$.  

By \eqref{E:CShypo} and Cotlar--Stein, 
\begin{align*}
&\|f\|_{L^2_\xi}^2 \gtrsim \|\sum_{j =1}^{2^{k_1(J-2)}} T_j f\|_{L^2_\xi}^2 
= \sum_{j =1}^{2^{k_1(J-2)}}  \|T_j f\|_{L^2_\xi}^2 + \sum_{1 \leq j \neq k \leq 2^{k_1(J-2)}} \langle T_j f,T_k f\rangle \\
&\qquad \geq \sum_{j =1}^{2^{k_1(J-2)}}  \|T_j f\|_{L^2_\xi}^2 -C_M \sum_{j =1}^{2^{k_1(J-2)}}  \sum_{k \neq j} R^{-M}|j-k|^{-M}\|f\|_{L^2_\xi}^2 \\
& \qquad \geq \sum_{j =1}^{2^{k_1(J-2)}}  \|T_j f\|_{L^2_\xi}^2 - C_M2^{k_1(J-2)}R^{-M}\|f\|_{L^2_\xi}^2.
\end{align*}
Using our lower bound $R \gtrsim 2^{k_1(J-2)\delta}$, we obtain \eqref{E:CotlarStein}.  
\end{proof}

\subsection*{Notation} We recycle notation, and will say for the remainder of the article that $A \lessapprox B$ if $A \lesssim_{\eps} 2^{\eps k_1}R^\eps B$ for all $\eps > 0$.  

Thus we want to show that 
\begin{equation} \label{E:local approx}
\|\mathcal E_1 f_1 \mathcal E_2 f_2\|_{L^{\frac{d+3}{d+1}}_{t,x}(Q_R)} \lessapprox (R^{(1-\eps)\alpha}+R^{C\eps})2^{k_1\delta} 2^{\frac{k_1(J-2)(d-1)}{2(d+3)}}\|f_1\|_{L^2_\xi}\|f_2\|_{L^2_\xi},
\end{equation}
and we assume (for the remainder of the argument) that $\scriptR(2\times2 \to \tfrac{d+3}{d+1};\delta,\alpha)$ holds, that $R \gtrsim 2^{k_1(J-2)\delta}$, and $\|f_1\|_{L^2_\xi} = \|f_2\|_{L^2_\xi} = 1$.

\section{Wave packet decomposition} \label{S:packet}

We recall that 
$$
Q_R := \{(t,x) : \tfrac12 2^{k_1(J-2)}R \leq t \leq 2^{k_1(J-2)}R, \: |x| \leq R\}.
$$
Define
$$
X_j := R^{\frac12} \Z^d, \qquad \Xi_j := (\R^{-\frac12}\Z^d) \cap B(0,4 c_0 2^{-k_j}), \qquad V_j := \nabla h_j(\Xi_j), \qquad j=1,2.
$$
For $j=1,2$, an $S_j$-tube is a set of the form
$$
T_j = \{(t,x) : |x-x_j(T_j) + tv_j(T_j)| < R^{\frac12}\},
$$
where $x_j(T_j) \in X_1$ and $v_j(T_j) \in V_j$.  

\begin{proposition} \label{P:packet}
There exist coefficients $(c_{T_j})$ and wave packets $(\phi_{T_j})$, indexed in those $S_j$-tubes $T_j$ satisfying $\dist(T_j,Q_R) \lesssim R$, such that for any $M > 0$,
\begin{equation}
\label{E:decomp}
\|\scriptE_1 f_1 \scriptE_2f_2\|_{L^\frac{d+3}{d+1}_{t,x}(Q_R)} \lesssim \|\sum_{T_1} c_{T_1}\phi_{T_1} \sum_{T_2} c_{T_2}\phi_{T_2}\|_{L^\frac{d+3}{d+1}_{t,x}(Q_R)} + O(1).
\end{equation}
Furthermore, the following hold for each $j=1,2$ and every tube $T_j$ appearing in the sum:
\begin{gather}
\label{E:ell2}
\|(c_{T_j})\|_{\ell^2_{T_j}} \lesssim 1\\
\label{E:free wave}
\phi_{T_j} = \scriptE_j \hat{\phi_{T_j}(0,\cdot)}\\
\label{E:support tube}
\supp \hat{\phi_{T_j}(0,\cdot)} \subseteq \{|\xi-\xi_j(T_j)| \lesssim R^{-1/2}\}\\
\label{E:decay tube}
|\phi_{T_j}(t,x)| \lesssim R^{-\frac d4}(1+\tfrac{|x-x_j(T_j) + tv_j(T_j)|}{R^{1/2}})^{-M}, \qquad (t,x) \in Q_R\\
\label{E:ortho tube}
\|\sum_{T_j} c_{T_j}' \phi_{T_j}(t,\cdot)\|_{L^2_x} \lesssim \|c_{T_j}'\|_{\ell^2_{T_j}}, \quad \ctc{for all} (c_{T_j}') \in \ell^2_{T_j}, \: t \in \R.
\end{gather}
\end{proposition}

The proof of this proposition will occupy the remainder of the section.  

We begin with the decomposition of $\scriptE_2f_2$.  Heuristically, an $S_2$-wave packet is concentrated on a tube that is transverse to the long axis of $Q_R$, so on $Q_R$ it should be concentrated on a tube of diameter $R^{\frac12}$ and length $R$.  Unfortunately, this heuristic neglects the role of dispersion, which means that we cannot simply decompose the ``initial data'' $\scriptE_2 f_2(0,\cdot)$ into pieces with Fourier support on $R^{-\frac12}$ balls and spatial concentration on $R^{\frac12}$ balls, and then propagate that decomposition forward.  Instead, we will apply Tao's elliptic wave packet decomposition \cite{TaoParab} to $\scriptE_2 f_2^{(j)}$ on $Q_j'$.  The precise statement we need is as follows.

\begin{lemma}[\cite{TaoParab}] \label{L:time loc packet 2}
For each $0 \leq j \leq 2^{k_1(J-2)}$, there exist coefficients $(c_{T_2}^{(j)})$ and wave packets $(\phi_{T_2}^{(j)})$, indexed in those tubes $T_2$ with $\dist(T_2,Q_j') \lesssim R$, that satisfy \emph{(\ref{E:free wave}-\ref{E:ortho tube})}, with the superscripts $(j)$ inserted, as well as 
\begin{gather} 
\label{E:time loc ell22}
\|(c_{T_2}^{(j)})\|_{\ell^2_{T_2}} \lesssim \|f_2^{(j)}\|_{L^2_\xi},\\
\label{E:time loc decomp2}
\scriptE_2 f_2^{(j)}(t,x) = \sum c_{T_2}^{(j)} \phi_{T_2}^{(j)} + O(R^{-M}), \qquad M > 0, \: (t,x) \in Q_j'.
\end{gather}
\end{lemma}

\begin{proof} 
For $j=0$, this follows from the wave packet decomposition in \cite{TaoParab}.  Given any $1 \leq j \leq 2^{k_1(J-2)}$, we may decompose
$$
\scriptE_2[e^{iRjh_2(\xi)} f_2^{(j)}](t,x) = \sum_{T_2} c_{T_2} \phi_{T_2} + O(R^{-M}), \ctc{on} Q_0'.
$$
Now we translate.  Our constants are the same: $c_{T_2}^{(j)} := c_{T_2}$, but our wave packets are shifted: $\phi_{T_2}^{(j)}(t,x) := \phi_{T_2}(t-Rj,x)$.  Thus $\phi_{T_2}^{(j)}$ is associated to a tube with parameters $x_j-Rj\nabla h_j(\xi_2)$ and $\xi_2$, where $x_2,\xi_2$ are the parameters for $\phi_{T_2}$.  The conclusions claimed in the lemma are then immediate from those obtained in the case $j=0$, and we are done.  
\end{proof}

Now let $\Lambda \subseteq \{0,1,\ldots,2^{k_1(J-2)}\}$ be a $C$-separated set for some sufficiently large $C$.  Applying the decomposition in Lemma~\ref{L:time loc packet 2} to each of the functions $f_2^{(j)}$, and then using the estimate in Lemma~\ref{L:f2j error}, together with the assumption $R \gtrsim 2^{\delta k_1}$, we obtain
\begin{equation} \label{E:decomp2}
\scriptE_2 f_2(t,x) = \sum_{T_2} c_{T_2}\phi_{T_2} + O(R^{-M}), \qquad (t,x) \in \bigcup_{j \in \Lambda} Q_j',
\end{equation}
where the tubes appearing in the sum all lie within a distance $O(R)$ of one of the $Q_j'$ with $j \in \Lambda$.  The conclusions (\ref{E:free wave}-\ref{E:decay tube}) follow immediately from Lemma~\ref{L:time loc packet 2}.  Inequality \eqref{E:ell2} just follows from \eqref{E:time loc ell22} and Lemma~\ref{L:CotlarStein}, and finally, \eqref{E:ortho tube} is just a consequence of the corresponding conclusion (with superscripts $(j)$ inserted) in Lemma~\ref{L:time loc packet 2}.  

Now we turn to the wave packet decomposition of $\scriptE_1 f_1$, which is essentially a rescaling of the elliptic case.

\begin{lemma} \label{L:packet1}
There exist coefficients $(c_{T_1})$ and wave packets $(\phi_{T_1})$, indexed in those tubes with $\dist(T_1,Q_R) \lesssim R$ and satisfying \emph{(\ref{E:ell2}-\ref{E:ortho tube})}, as well as
\begin{equation} \label{E:decomp1}
\scriptE_1 f_1(t,x) = \sum_{T_1} c_{T_1} \phi_{T_1} + O(R^{-M}), \qquad (t,x) \in Q_R.
\end{equation}
\end{lemma}

\begin{proof}
This may be obtained by rescaling the standard wave packet decomposition from \cite{TaoParab}.
\end{proof}

From here, the proof of Propositon~\ref{P:packet} is quick.  

\begin{proof}[Proof of Proposition~\ref{P:packet}]
 Given a $C$-separated subset $\Lambda\subseteq\{0,\ldots,2^{k_1(J-2)}\}$, let $(c_{T_2}^{\Lambda})$, $(\phi_{T_2}^{\Lambda})$ denote the coefficients and wave packets appearing in \eqref{E:decomp2}.  Then
$$
|\scriptE_2 f_2(t,x)| \leq \sum_{\Lambda} |\sum_{T_2} c_{T_2}^{\Lambda} \phi_{T_2}^{\Lambda}|+O(R^{-M}), \qquad (t,x) \in Q_R,
$$
where the sum is taken over a disjoint collection of $C$ such $\Lambda$'s.  Combining this with the wave packet decomposition in Lemma~\ref{L:packet1} and the fact that $|\scriptE_j f_j| \lesssim 1$ (because $\|f_j\|_{L^1_\xi} \lesssim \|f_j\|_{L^2_\xi} = 1$)
$$
|\scriptE_1 f_1 \scriptE_2 f_2(t,x)| \leq \sum_{\Lambda} |\sum_{T_1} c_{T_1} \phi_{T_1} \sum_{T_2}c_{T_2}^{\Lambda} \phi_{T_2}^{\Lambda}| + O(R^{-M}).
$$
The estimate \eqref{E:decomp} follows from H\"older, the triangle inequality, the pigeonhole principle, which lets us pick a single $\Lambda$, and $R \gtrsim 2^{\delta k_1}$.  The properties (\ref{E:ell2}-\ref{E:ortho tube}) have already been established, so we are done.  
\end{proof}

\section{The local and global terms}

The wave packet decomposition allows for a number of reductions.  These follow the general scheme of \cite{TaoParab}, but modifications are needed throughout to account for the degeneracy.

First, it suffices to show
\begin{equation} \label{E:big bad sum}
\|(\sum_{T_1} c_{T_1} \phi_{T_1})(\sum_{T_2} c_{T_2} \phi_{T_2})\|_{L^{\frac{d+3}{d+1}}_{t,x}(Q_R)} \lessapprox 2^{k_1\delta}(R^{(1-\eps)\alpha}+R^{C\eps})2^{\frac{k_1(J-2)(d-1)}{2(d+3)}},
\end{equation}
whenever the sums are taken over $S_j$-tubes $T_j$ with $\dist(T_j,Q_R) \lesssim R$, $\|(c_{T_j})\|_{\ell^2_{T_j}} \lesssim 1$, and the wave packets are as described in Proposition~\ref{P:packet}.  We only sum over $O(R^{\frac d2})$ $S_1$-tubes and $O(2^{k_1(J-2)}R^{\frac d2})$ $S_2$-tubes, so we may assume that for each $T_j$ in the sum, $|c_{T_j}| \gtrsim R^{-c_d}2^{-k_1(J-2)c_d}$.  This leaves $O(k_1 \log R)$ possible dyadic values for $c_{T_j}$ and by pigeonholing, it suffices to prove
\begin{equation} \label{E:reduced sum}
\begin{aligned}
&\|(\sum_{T_1 \in \mathcal T_1} \phi_{T_1})(\sum_{T_2 \in \mathcal T_2} \phi_{T_2})\|_{L^{\frac{d+3}{d+1}}_{t,x}(Q_R)} \\
&\qquad\qquad \lessapprox 2^{k_1\delta}(R^{(1-\eps)\alpha}+R^{C\eps})2^{\frac{k_1(J-2)(d-1)}{2(d+3)}}(\#\mathcal T_1\#\mathcal T_2)^{\frac12},
\end{aligned}
\end{equation}
whenever each $\mathcal T_j$ is a collection of $S_j$-tubes $T_j$ with $\dist(T_j,Q_R) \lesssim R$.  

We decompose $Q_R = \bigcup_{B \in \mathcal B} B$, where $\mathcal B$ is a collection of finitely overlapping translates of $R^{-\eps}Q_R$.  We also make a second, finer decomposition $Q_R = \bigcup_{q \in \scriptQ} q$, where $\scriptQ$ is a collection of finitely overlapping $R^{1/2}$ balls.  For $q \in \scriptQ$, define
$$
\mathcal T_j(q)= \{T_j \in \scriptT_j : T_j \cap R^\eps q \neq \emptyset\}.
$$
Given dyadic values $1 \leq \mu_1,\mu_2,\lambda_1,\lambda_2 \lesssim 2^{k_1(J-2)}R^{2(1+d)}$, define
\begin{align}\label{E:mu1mu2}
\scriptQ(\mu_1,\mu_2) &= \{q \in \scriptQ : \tfrac12 \mu_j \leq \#\scriptT_j(q) \leq \mu_j, \: j=1,2\},\\
\label{E:lambda1lambda2}
\scriptT_j(\lambda_j,\mu_1,\mu_2) &= \{T_j \in \scriptT_j : \tfrac12 \lambda_j \leq \#\{q \in \scriptQ(\mu_1,\mu_2) : T_j \in \scriptT_j(q)\} \leq \lambda_j\},\\
\label{E:B1B2}
B_j(T_j,\lambda_j,\mu_1,\mu_2) &= \mathop{\rm{arg\, max}}_{B \in \scriptB} \#\{q \in \scriptQ(\mu_1,\mu_2) : T_j \in \scriptT_j(q) \ctc{and} q \cap B \neq \emptyset\}.
\end{align}
If $B \in \scriptB$ and $T_j \in \scriptT_j$, say $T_j \sim_{\lambda_j,\mu_1,\mu_2} B$ if $B \subseteq C B_j(T_j,\lambda_j,\mu_1,\mu_2)$ and say $T_j \sim B$ if $T_j \sim_{\lambda_j,\mu_1,\mu_2} B$ for some $\lambda_j,\mu_1,\mu_2$.  (Here $C$ is sufficiently large for the proof of Lemmas~\ref{L:geom combinat1} and~\ref{L:geom combinat2} in Section~\ref{S:combinat}.)  Finally, given $B$, let $\mathcal T_j^\sim(B) = \{T_j \in \scriptT_j : T_j \sim B\}$, $\scriptT_j^{\not\sim}(B) = \scriptT_j \setminus \scriptT_j^\sim(B)$.  

By the triangle inequality,
\begin{align*}
\|(\sum_{T_1 \in\scriptT_1} \phi_{T_1})(\sum_{T_2 \in \scriptT_2} \phi_{T_2})\|_{L^{\frac{d+3}{d+1}}_{t,x}(Q_R)}
\leq
&\sum_{B \in \scriptB} \|(\sum_{\scriptT_1^\sim(B)} \phi_{T_1})(\sum_{\scriptT_2^\sim(B)} \phi_{T_2})\|_{L^{\frac{d+3}{d+1}}_{t,x}(B)}\\
&  +
\sum_{B \in \scriptB} \|(\sum_{\in \scriptT_1^{\not\sim}(B)} \phi_{T_1})(\sum_{\scriptT_2^\sim(B)} \phi_{T_2})\|_{L^{\frac{d+3}{d+1}}_{t,x}(B)}\\
&+
\sum_{B \in \scriptB} \|(\sum_{\scriptT_1^{\sim}(B)} \phi_{T_1})(\sum_{\scriptT_2^{\not\sim}(B)} \phi_{T_2})\|_{L^{\frac{d+3}{d+1}}_{t,x}(B)}\\
&+
\sum_{B \in \scriptB} \|(\sum_{\scriptT_1^{\not\sim}(B)} \phi_{T_1})(\sum_{\scriptT_2^{\not\sim}(B)} \phi_{T_2})\|_{L^{\frac{d+3}{d+1}}_{t,x}(B)}.
\end{align*}
As in \cite{TaoParab}, we will think of the first as the ``local term,'' and the last three as ``global.''  

The local term may be bounded easily using the induction hypothesis and the fact that there are only $O(\log R)$ possible dyadic values of $\lambda_1,\lambda_2,\mu_1,\mu_2$:
\begin{align*}
&\sum_{B \in \scriptB}\|(\sum_{\scriptT_1^\sim(B)} \phi_{T_1})(\sum_{\scriptT_2^\sim(B)} \phi_{T_2})\|_{L^{\frac{d+3}{d+1}}_{t,x}(B)}\\
&\qquad\qquad\lessapprox \sum_{B \in \scriptB} 2^{k_1\delta}R^{(1-\eps)\alpha}2^{\frac{k_1(J-2)(d-1)}{2(d+3)}}(\#\scriptT_1^\sim(B)\#\scriptT_2^\sim(B))^{\frac12} \\
&\qquad\qquad \leq 2^{k_1\delta}R^{(1-\eps)\alpha} 2^{\frac{k_1(J-2)(d-1)}{2(d+3)}} \bigl(\sum_{B \in \scriptB} \#\scriptT_1^\sim(B)\bigr)^{\frac12}\bigl(\sum_{B \in \scriptB}\#\scriptT_2^\sim(B))^{\frac12}\\
&\qquad\qquad= 2^{k_1\delta}R^{(1-\eps)\alpha}2^{\frac{k_1(J-2)(d-1)}{2(d+3)}}\bigl(\sum_{T_1 \in \scriptT_1} \sum_{B:B \sim T_1} 1\bigr)^{\frac12} \bigl(\sum_{T_2 \in \scriptT_2} \sum_{B:B \sim T_2}1\bigr)^{\frac12} \\
&\qquad\qquad\lessapprox 2^{k_1\delta}R^{(1-\eps)\alpha}2^{\frac{k_1(J-2)(d-1)}{2(d+3)}}(\#\scriptT_1\#\scriptT_2)^{\frac12}.
\end{align*}

It remains to control the global terms.  


\section{Reduction to two combinatorial estimates}


It suffices to show that for each $B \in \scriptB$,
\begin{equation} \label{E:global bound}
\|\sum_{T_1 \in \scriptT_1'(B)} \sum_{T_2 \in \scriptT_2'(B)} \phi_{T_1} \phi_{T_2}\|_{L^{\frac{d+3}{d+1}}_{t,x}(B)} 
\lessapprox  2^{k_1\delta}R^{C\eps}2^{\frac{k_1(J-2)(d-1)}{2(d+3)}}(\#\scriptT_1\#\scriptT_2)^{\frac12},
\end{equation}
in each of the cases $\scriptT_1'(B) = \scriptT_1^{\not\sim}(B)$ and $\scriptT_2'(B) \subseteq \scriptT_2$;  $\scriptT_1'(B) \subseteq \scriptT_1$ and $\scriptT_2'(B) = \scriptT_2^{\not\sim}(B)$.  The arguments for the different cases will only diverge in the proofs of the combinatorial estimates.  For convenience, we will use the notation $\scriptT_j'(\cdot)$ (with various arguments within the parentheses)  to refer to $\scriptT_j$ or $\scriptT_j^{\not\sim}$, depending on which case we are in.  

\begin{lemma}\label{L:L1 est}
\begin{equation} \label{E:L1 est}
\|\sum_{T_1 \in \scriptT_1'(B)}\sum_{T_2 \in \scriptT_2'(B)} \phi_{T_1}\phi_{T_2}\|_{L^1_{t,x}(B)} \lesssim 2^{\frac{k_1(J-2)}2}R (\#\scriptT_1'(B)\#\scriptT_2'(B))^{1/2}.
\end{equation}
\end{lemma}

\begin{proof}
We begin by estimating the contributions from the $S_j$-tubes separately in the cases $j=1,2$.  

By H\"older's inequality, \eqref{E:free wave}, and \eqref{E:ortho tube},
\begin{equation} \label{E:L1 est T1}
\|\sum_{T_1 \in \scriptT_1'(B)} \phi_{T_1}\|_{L^2_{t,x}(B)} \lesssim 2^{\frac{k_1(J-2)}2}R^{1/2}\|\sum_{T_1 \in \scriptT_1'(B)} \phi_{T_1}(0)\|_{L^2_x} \lesssim 2^{\frac{k_1(J-2)}2}R^{1/2}(\#\scriptT_1)^{1/2}.
\end{equation}

Write $B = \bigcup_{j=0}^{2^{k_1(J-2)}} B_j$, where each $B_j$ is an $R^{1-\eps}$ cube, and for each $j$, let $\scriptT_2'(B_j)$ denote the set of tubes $T_2 \in \scriptT_2'(B)$ for which $\dist(T_2,B_j) \lesssim R^{1-\eps}$.  Note that each tube is in $\scriptT_2(B_j)$ for $O(1)$ values of $j$ by transversality of $S_2$.  Using the decay estimate \eqref{E:decay tube}, H\"older and the fact that $R \gtrsim 2^{\delta c_d k_1}$, \eqref{E:free wave}, \eqref{E:ortho tube}, and the near disjointness of the sets $\scriptT_2'(B_j)$, 
\begin{align} \notag
&\|\sum_{T_2 \in \scriptT_2'(B)}\phi_{T_2}\|_{L^2_{t,x}(B)}^2 
= \sum_{j=0}^{2^{k_1(J-2)}}\|\sum_{T_2 \in \scriptT_2'(B)}\phi_{T_2}\|_{L^2_{t,x}(B_j)}^2 \\\notag
&\qquad = \sum_{j=0}^{2^{k_1(J-2)}}\|\sum_{T_2 \in \scriptT_2'(B_j)}\phi_{T_2} + O(R^{-M})\|_{L^2_{t,x}(B_j)}^2 \\\label{E:L1 est T2}
&\qquad \lesssim 1+R\sum_{j=0}^{2^{k_1(J-2)}}\|\sum_{T_2 \in \scriptT_2'(B_j)}\phi_{T_1}(0)\|_{L^2_x}^2 \lesssim R\sum_{j=0}^{2^{k_1(J-2)}}\#\scriptT_2'(B_j) \lesssim R\#\scriptT_2'(B).
\end{align}

Finally, \eqref{E:L1 est} just follows from \eqref{E:L1 est T1}, \eqref{E:L1 est T2}, and Cauchy--Schwarz.  
\end{proof}

By interpolation, \eqref{E:global bound} will then follow from the estimate
\begin{equation} \label{E:global L2}
\|\sum_{\scriptT_1'(B)} \sum_{\scriptT_2'(B)} \phi_{T_1}\phi_{T_2}\|_{L^2_{t,x}(B)} \lessapprox 2^{k_1\delta}R^{C\eps}R^{-\frac{d-1}4}(\#\scriptT_1\#\scriptT_2)^{1/2}.
\end{equation}

We decompose
\begin{equation} \label{E:decomposed}
\|\sum_{\scriptT_1'(B)}\sum_{\scriptT_2'(B)} \phi_{T_1}\phi_{T_2}\|_{L^2_{t,x}(B)}^2 \leq \sum_{q \subseteq 2 B}\| \sum_{\scriptT_1'(B)}\sum_{\scriptT_2'(B)} \phi_{T_1}\phi_{T_2}\|_{L^2_{t,x}(q)}^2.
\end{equation}
By the decay estimate, if $T_j \not\in \scriptT_j(q)$ (i.e.\ $T_j \cap R^\eps q = \emptyset$), $|\phi_{T_j}| \lesssim R^{-M}$ on $q$, for arbitrarily large $M$, so the contribution from any tubes not in $\scriptT_j(q)$ is negligible.  By this and pigeonholing, it suffices to prove that
\begin{equation} \label{E:refined}
\sum_{q \in \scriptQ(\mu_1,\mu_2)}\|\sum_{\scriptT_1'(q)}\sum_{\scriptT_2'(q)} \phi_{T_1}\phi_{T_2}\|_{L^2_{t,x}(q)}^2 \lessapprox R^{C\eps - \frac{d-1}2}\#\scriptT_1\#\scriptT_2,
\end{equation}
where
\begin{equation}\label{E:Tj'q}
\scriptT_j'(q) = \scriptT_j'(B) \cap \scriptT_j(q) \cap \scriptT_j(\lambda_j,\mu_1,\mu_2),
\end{equation}
and $1 \leq \mu_1,\mu_2,\lambda_1,\lambda_2 \lesssim R^{100d}$ are arbitrary dyadic values, which will remain fixed for the remainder of the section.

Given $\xi_1 \in B(0,2^{-k_1+1}c_0)$ and $\xi_2' \in B(0,2c_0)$, or $\xi_1' \in B(0,2^{-k_1+1}c_0)$ and $\xi_2 \in B(0,2c_0)$ (respectively), the functions
\begin{equation}\label{E:level set is pi}
\begin{aligned}
\xi_1' &\mapsto (h_1(\xi_1) + h_2(\xi_1'+\xi_2'-\xi_1)) - (h_1(\xi_1') + h_2(\xi_2'))\\
\xi_2' &\mapsto (h_1(\xi_1'+\xi_2'-\xi_2) + h_2(\xi_2)) - (h_1(\xi_1') + h_2(\xi_2'))
\end{aligned}
\end{equation}
have gradients comparable to 1, so the hypersurfaces
\begin{align*}
\pi_1(\xi_1,\xi_2') = \{\xi_1'\in B(0,2c_0) : h_1(\xi_1) + h_2(\xi_1'+\xi_2'-\xi_1) = h_1(\xi_1') + h_2(\xi_2')\},\\
\pi_2(\xi_2,\xi_1') = \{\xi_2'\in B(0,2c_0) : h_1(\xi_1'+\xi_2'-\xi_2) + h_2(\xi_2) = h_1(\xi_1') + h_2(\xi_2')\},
\end{align*}
are smoothly embedded.  

Given $\xi_1,\xi_1' \in \Xi_1$ and $\xi_2,\xi_2' \in \Xi_2$, define collections
\begin{equation} \label{E:pi tubes}
\begin{aligned}
\scriptT_1'(q,\xi_1,\xi_2') &= \{T_1' \in \scriptT_1'(q) : \dist(\xi(T_1'),\pi_1(\xi_1,\xi_2')) \lesssim R^{C\eps-1/2}\}\\
\scriptT_2'(q,\xi_2,\xi_1') &= \{T_2' \in \scriptT_2'(q) : \dist(\xi(T_2'),\pi_1(\xi_2,\xi_1')) \lesssim R^{C\eps-1/2}\},
\end{aligned}
\end{equation}
and quantities
\begin{equation} \label{E:nuj}
\begin{aligned}
\nu_1(q) = \sup_{\xi_1 \in \Xi_1,\xi_2' \in \Xi_2} \#\scriptT_1'(q,\xi_1,\xi_2')\\
\nu_2(q) = \sup_{\xi_2 \in \Xi_2,\xi_1' \in \Xi_1} \#\scriptT_2'(q,\xi_2,\xi_2').
\end{aligned}
\end{equation}

\begin{lemma}\label{L:combinatorial to L2}
For any $q \in \scriptQ(\mu_1,\mu_2)$, and $j=1,2$,
\begin{equation} \label{E:combinatorial to L2}
\|\sum_{\scriptT_1'(q)}\sum_{\scriptT_2'(q)} \phi_{T_1}\phi_{T_2}\|_{L^2_{t,x}(q)}^2 
\lessapprox R^{C\eps}R^{-(d-1)/2}\nu_j(q)\#\scriptT_1(q)\#\scriptT_2(q)
\end{equation}
\end{lemma}

\begin{proof}
We give the proof when $j=1$.  
By simple arithmetic,
\begin{align*}
\|\sum_{\scriptT_1'(q)}\sum_{\scriptT_2'(q)} \phi_{T_1}\phi_{T_2}\|_{L^2_{t,x}(q)}^2 \leq \sum_{T_1,T_1' \in \scriptT_1'(q)}\sum_{T_2,T_2' \in \scriptT_2'(q)} \langle \phi_{T_1}\phi_{T_2},\phi_{T_1'}\phi_{T_2'}\rangle.
\end{align*}

By Plancherel, $\langle \phi_{T_1}\phi_{T_2},\phi_{T_1'}\phi_{T_2'}\rangle$ equals zero unless
\begin{equation}\label{E:defining pi}
\begin{gathered}
\xi(T_1) + \xi(T_2) = \xi(T_1') + \xi(T_2') + O(R^{-1/2})\\
h_1(\xi(T_1)) + h_2(\xi(T_2)) = h_1(\xi(T_1')) + h_2(\xi(T_2')) + O(R^{-1/2}),
\end{gathered}
\end{equation}
i.e.\ unless $\dist(T_1',\pi_1(\xi_1(T_1),\xi_2(T_2'))), \dist(T_2',\pi_2(\xi_2(T_2),\xi_1(T_1'))) \lesssim R^{-1/2}$.  

By Plancherel, a simple change of variables using transversality of the surfaces $S_1,S_2$, and the small frequency support of the $\phi_j(0)$,
$$
\|\phi_{T_1}\phi_{T_2}\|_{L^2_{t,x}} = \|\scriptF(\phi_{T_1})*\scriptF(\phi_{T_2})\|_{L^2_{\tau,\xi}} \lesssim R^{-(d-1)/4}\|\hat{\phi_{T_1}(0)}\|_{L^2_\xi} \|\hat{\phi_{T_2}(0)}\|_{L^2_\xi} \sim R^{-(d-1)/4}.
$$

We claim that, given $T_1,T_1',T_2$, \eqref{E:defining pi} can hold for at most $O(R^{C\eps})$ tubes $T_2$ in $\scriptT_2'(q)$.  Indeed, the second map in \eqref{E:level set is pi} has gradient comparable to 1, so the equations \eqref{E:defining pi} essentially determine $\xi_2(T_2')$; when combined with $q$, this direction determines $T_2'$.  

Putting these observations together, 
\begin{align*}
&\|\sum_{\scriptT_1'(q)}\sum_{\scriptT_2'(q)} \phi_{T_1}\phi_{T_2}\|_{L^2_{t,x}(q)}^2 
\lesssim \sum_{T_1 \in \scriptT_1'(q)}\sum_{T_2' \in \scriptT_2'(q)}\sum_{T_1 \in \scriptT_1'(q,\xi_1(T_1),\xi_2(T_2'))} R^{C\eps-(d-1)/2} \\
&\qquad \qquad \lessapprox R^{C\eps-(d-1)/2}\#\scriptT_1'(q)\#\scriptT_2'(q)\nu_1(q).
\end{align*}
The proof that 
$$
\|\sum_{\scriptT_1'(q)}\sum_{\scriptT_2'(q)} \phi_{T_1}\phi_{T_2}\|_{L^2_{t,x}(q)}^2 \lessapprox R^{C\eps-(d-1)/2}\#\scriptT_1'(q)\#\scriptT_2'(q)\nu_2(q)
$$
is exactly the same.
\end{proof}

It remains to control the sum on $q$ of the right side of \eqref{E:combinatorial to L2}.  We will show that if $\scriptT_1' = \scriptT_1^{\not\sim}$, then 
\begin{align} \label{E:combinatorial est nu1}
\sum_{\scriptQ(\mu,B)} \#\scriptT_1^{\not\sim}(q) \#\scriptT_2'(q) \nu_1(q) &\lessapprox R^{C\eps}\#\scriptT_1\#\scriptT_2
\end{align}
and that if $\scriptT_2' = \scriptT_2^{\not\sim}$, then
\begin{align}
\label{E:combinatorial est nu2}
\sum_{\scriptQ(\mu,B)} \#\scriptT_1'(q) \#\scriptT_2^{\not\sim}(q) \nu_1(q) &\lessapprox R^{C\eps}\#\scriptT_1\#\scriptT_2,
\end{align}
where $\scriptT_j'(q)$ is as in \eqref{E:Tj'q} and $\scriptQ(\mu,B) = \{q \in \scriptQ(\mu_1,\mu_2): q \subseteq 2B\}$.  These are our combinatorial estimates.

\section{Proofs of the combinatorial estimates} \label{S:combinat}

This section will be devoted to the proofs of the combinatorial estimates \eqref{E:combinatorial est nu1} and \eqref{E:combinatorial est nu2}.  There are some differences in the proofs due to the differing geometries of the intersections of $S_j$ tubes with $Q_R$ for $j=1,2$, but the two inequalities are more similar than not.  We will begin with \eqref{E:combinatorial est nu1} and indicate the changes necessary for \eqref{E:combinatorial est nu2}.  The argument is adapted from that of \cite{TaoParab}, so we will be somewhat brief.  

Recalling the role played by $\mu$ from \eqref{E:mu1mu2}, using Fubini, and then recalling the role of $\lambda$ from \eqref{E:lambda1lambda2} and \eqref{E:Tj'q},
\begin{align*}
&\sum_{q \in \scriptQ(\mu,B)} \#\scriptT_1^{\not\sim}(q)\#\scriptT_2'(q) \nu_1 
\lesssim \mu_2\nu_1 \sum_{q \in \scriptQ(\mu,B)} \#\scriptT_1^{\not\sim}(q) \\
&\qquad \qquad = \mu_2\nu_1 \sum_{T_1 \in \scriptT_1^{\not\sim}(B)} \#\{q \in \scriptQ(\mu,B) : T_1 \in \scriptT_1(B)\} \lesssim \mu_2\nu_1\lambda_1 \#\scriptT_1^{\not\sim}(B).
\end{align*}
Thus \eqref{E:combinatorial est nu1} will be proven if we can show that for an arbitrary (henceforth fixed) $q_0 \in \scriptQ(\mu,B)$ and arbitrary (also fixed) $\xi_1 \in \Xi_1, \xi_2' \in \Xi_2$,
\begin{equation} \label{E:q0 est}
\#\scriptT_1^{\not\sim}(q_0,\xi_1,\xi_2') \lessapprox 2^{k_1\delta} R^{C\eps} \frac{\#\scriptT_2}{\mu_2\lambda_1}.
\end{equation}

If $T_1 \in \scriptT_1^{\not\sim}(q_0,\xi_1,\xi_2')$, $B \not\subseteq C B_1(T_1,\lambda_1,\mu_1,\mu_2)$, so
$$
\#\{q \in \scriptQ(\mu) : T_1 \cap R^\eps q \neq \emptyset, \: \: q \cap \tfrac12 C B = \emptyset\} \gtrsim R^{-C\eps}\lambda_1.
$$
Furthermore, if $q \in \scriptQ(\mu)$, $\#\{T_2 \in \scriptT_2 : T_2 \cap R^\eps q \neq \emptyset\} \gtrsim \mu_2$, so
\begin{align*}
&\#\{(q,T_1,T_2) \in \scriptQ(\mu) \times \scriptT_1^{\not\sim}(q_0,\xi_1,\xi_2') \times \scriptT_2 : T_1 \cap R^\eps q \neq \emptyset, \:T_2 \cap R^\eps q \neq \emptyset, \: q \cap \tfrac12C B = \emptyset\}\\
&\qquad\qquad\qquad \gtrsim R^{-C\eps} \lambda_1 \mu_2 \#\scriptT_1^{\not\sim}(q_0,\xi_1,\xi_2').
\end{align*}
Thus it suffices to show that the left side of this inequality is bounded ($\lessapprox$) by $2^{k_1\delta} R^{C\eps}\#\scriptT_2$.  This will follow from the next lemma.

\begin{lemma}\label{L:geom combinat1}
If $T_2 \in \scriptT_2$,
$$
\#\{(q,T_1) \in \scriptQ \times \scriptT_1^{\not\sim}(q_0,\xi_1,\xi_2') : T_1 \cap R^\eps q \neq \emptyset, \: T_2 \cap R^\eps q \neq \emptyset, \: q \cap \tfrac{C}2 B = \emptyset\} \lessapprox 2^{k_1\delta} R^{C\eps}.
$$
\end{lemma}

\begin{proof}
Let $(t_0,x_0)$ and $(t,x)$ denote the centers of $q_0$ and $q$, respectively.  Suppose that the pair $(q,T_1)$ is in the set above.  Since $T_1 \cap R^\eps q_0, T_1 \cap R^\eps q \neq \emptyset$,
$$
x-x_0 = (t-t_0)v_1(T_1) + O(R^{1/2+\eps}),
$$
which implies that $|x-x_0| \lesssim 2^{-k_1(J-2)}|t-t_0| + O(R^{1/2+\eps})$.  On the other hand, $q_0 \subseteq 2B$ and $q \cap \tfrac{C}2 B = \emptyset$ together imply that $|t-t_0| \gtrsim 2^{k_1(J-2)}R^{1-\eps}$ or $|x-x_0| \gtrsim R^{1-\eps}$; by the preceding observation, the former must hold.  This implies two things.   

First, $(t,x)$ must lie within $O(R^{1/2+\eps})$ of the hypersurface $\Gamma+(t_0,x_0)$, where
$$
\Gamma = \Gamma(\xi_1,\xi_2') = \{(t,x) : t \gtrsim 2^{k_1(J-2)}R^{1-\eps}, \: x = t \nabla h_1(\xi_1') \: \text{for some}\: \xi_1' \in \pi_1(\xi_1,\xi_2')\}.
$$
We will show that $\Gamma$ is transverse to directions in $V_2$.  Assuming this for a moment, our tube $T_2$ intersects $\Gamma$ in a ball of radius $R^{1/2}$ and thus picks out $O(R^{C\eps})$ cubes $q$.

Second, $v_1(T_1) = \tfrac{x-x_0}{t-t_0} + O(2^{-k_1(J-2)}R^{-1/2+C\eps})$, so given $q$, there are at most $O(R^{C\eps})$ possible choices for $T_1$. 

The proof of the lemma will be complete once we verify the transversality.  By ellipticity of $g_1$, $\nabla h_1$ is an invertible function.  Unwinding the definitions,
$$
\Gamma = \{(t,x) : h_2(\xi_1+(\nabla h_1)^{-1}(\tfrac xt)-\xi_2') - h_1((\nabla h_1)^{-1}(\tfrac xt)) = h_1(\xi_1)-h_2(\xi_2')\}.
$$
Thus (undoing the scalings), the normal at $(t,x)$ is parallel to the vector
\begin{align*}
&(-2^{-k_1(J-1)}\nabla g_2(\eta_2)(D^2 g_1(\eta_1'))^{-1}\nabla g_1(\eta_1')+2^{-2k_1(J-1)}\nabla g_1(\eta_1')(D^2 g_1(\eta_1'))^{-1}\nabla g_1(\eta_1'),\\
&\qquad\qquad \nabla g_2(\eta_2)(D^2g_2(\eta_1'))^{-1}-2^{-k_1(J-1)}\nabla g_1(\eta_1')(D^2g_1(\eta_1'))^{-1}),
\end{align*}
where $\frac xt = \nabla h_1(\xi_1') = 2^{-k_1(J-1)}\nabla g_1(\eta_1')$ and $\eta_2 = \xi_2 = \xi_1+\xi_1'-\xi_2'$ and $|\eta_1|,|\eta_2| < c_0$.  Recalling that $D^2 g_1$ is close to the identity and $\nabla g_2$ is close to $e_1$, we see that this normal makes a large angle with any $(1,-v_2(T_2))$, so we have the transversality we want.  
\end{proof}

This completes the proof of \eqref{E:combinatorial est nu1}.  Now we turn to \eqref{E:combinatorial est nu2}.  Simply changing subscripts in the earlier argument, we can reduce matters to proving the following.
\begin{lemma}\label{L:geom combinat2}
If $T_1 \in \scriptT_1$,
$$
\#\{(q,T_2) \in \scriptQ \times \scriptT_2^{\not\sim}(q_0,\xi_1,\xi_2') : T_2 \cap R^{\eps}q \neq \emptyset, \: T_1 \cap R^\eps q \neq \emptyset, \: q \cap \tfrac{C}2 B = \emptyset\} \lessapprox 2^{k_1\delta}R^{C\eps}.
$$
\end{lemma}

\begin{proof}
As before, let $(t_0,x_0), (t,x)$ denote the centers of $q_0,q$.  This time, if $(q,T_2)$ is in the above set, $T_2 \cap R^\eps q_0,T_2 \cap R^\eps q \neq \emptyset$, which implies that $|t-t_0| \lesssim R$.  Thus since $q_0 \subseteq 2B$ and $q \cap \tfrac C2 B = \emptyset$, $|x-x_0| \gtrsim R^{1-\eps}$.  Now $x-x_0 = (t-t_0)v_2(T_2) + O(R^{1/2+\eps})$, and since $|v_2(T_2)| \lesssim 1$, $|t-t_0| \gtrsim R^{1-\eps}$ as well.  

Now we know that $(t,x)$ must lie within $O(R^{1/2+\eps})$ of the hypersurface 
$$
\{(t,x) : |t-t_0| \gtrsim R^{1-\eps}, \: (x-x_0) = (t-t_0)\nabla h_2(\xi_2'), \: \text{for some}\: \xi_2' \in \pi_2(\xi_2,\xi_1')\}.
$$
It is similar (but slightly simpler) to show that this hypersurface is transverse to directions in $V_1$ (such directions are nearly vertical), so $T_1$ intersects it in a ball of radius $R^{1/2}$, picking out $O(R^{C\eps})$ cubes $q$.

Between the estimate $v_2(T_2) = \tfrac{x-x_0}{t-t_0} + O(R^{-1/2+C\eps})$ and the fact that $T_2$ intersects $R^\eps q$, there are only $O(R^{C\eps})$ possibilities for $T_2$ as well, so we are done.  
\end{proof}

\section{Extensions and remarks}

The same argument gives bounds for restriction to the graph of $a_1|\xi|^{k_1} + \cdots + a_n |\xi|^{k_n}$, for any coefficients $a_1,\ldots,a_n > 0$ and real powers $2 \leq k_1 < \cdots < k_n$; the coefficients however will depend on the $k_i$, not just on $k_n$.  

Let $P(t) = a_2 t^2 + \cdots + a_n t^n$, and assume that $P''(t) > 0$ for all $t > 0$.  Let $\nmin$ and $\nmax$ be the degrees of the lowest and highest (respectively) terms of $P$; their coefficients, $a_{\nmin}$ and $a_{\nmax}$, must be positive.  Let $\Imin = \{t \geq 0 : a_{\nmin}t^{\nmin} \geq \max_i |a_i t^i|\}$, $\Imax = \{t \geq 0 : a_{\nmax} t^{\nmax} \geq \max_i |a_i t_i|\}$, and $\Imed = [0,\infty) \setminus (\Imin\cup\Imax)$.  Then $\Imin$ contains all points sufficiently small and $\Imax$ all points sufficiently large, so $\{(P(|\xi|),\xi) : |\xi| \in \Imed\}$ is compact and elliptic.  The methods of the preceding sections apply on $\{(P(|\xi|),\xi) : |\xi| \in I_\bullet\}$ for $\bullet = \rm{min},\rm{max}$, and we can obtain a non-uniform version of Theorem~\ref{T:main} for restriction to the graph of $P(|\xi|)$.  Arguing similarly (but only separating out the low frequencies), we may prove such a nonuniform theorem for hypersurfaces of the form
$$
\{(\phi(|\xi|),\xi) : |\xi| \leq R\},
$$
whenever $\phi$ is smooth, $\phi'(0) = 0$, $\phi$ is finite type at 0, and $\phi''(t) > 0$ for $t > 0$.  It would be nice to know more uniform versions of these results.  

As a corollary of Theorem~\ref{T:main}, we can obtain an unweighted result, which is necessarily nonuniform.

\begin{corollary} \label{C:unweighted}
Let $P$ be a polynomial on $\R$ with $P'(0) = 0$ and $P''(t) > 0$ for all $t > 0$.  Let $\nmin$ denote the lowest nonzero power of $t$ appearing in $P$ and $\nmax$ the greatest.  Then, conditional on the restriction conjecture $\scriptR(p_0 \to q_0)$ for the admissible pair $(p_0,q_0)$, 
\begin{equation} \label{E:unweighted}
\|\hat f(P(|\xi|),\xi)\|_{L^r(d\xi)} \lesssim \|f\|_p, 
\end{equation}
provided $1 \leq p < p_0$ and either $r \geq p$ and $\tfrac{dp'}{\nmax + d} \leq r \leq \tfrac{dp'}{\nmin+d}$, or $r < p$ and $\tfrac{dp'}{\nmax+d} < r < \tfrac{dp'}{\nmin+d}$.  The implicit constant depends on $d,p,P$.  
\end{corollary}

For a given value of $p$ the range of $r$ in the corollary is sharp.  In particular, the full conjectured range of unweighted bounds would follow from a resolution of the restriction conjecture.  We note that in certain cases, some of the exponents $r$ covered in the corollary may be less than 1.  

The proof of the corollary uses an argument dating back at least to Drury--Marshall in \cite{DM85} and some simple observations.  

\begin{proof}
We give the proof when $\nmin < \nmax$.  In the monomial case $\nmin=\nmax$, the argument is similar but simpler.  

By Theorem~\ref{T:main} (or the extension mentioned above),
\begin{equation} \label{E:cor 1}
\|\hat f(P(|\xi|),\xi) \Lambda_P(\xi)^{\frac{d}{(d+2)p'}}\|_{L^q} \lesssim \|f\|_p, \qquad 1\leq p < p_0, \quad q = \tfrac{dp'}{d+2}.
\end{equation}
Let $\Imin,\Imed,\Imax$ be the intervals defined just before the statement of the corollary, and let $A_\bullet = I_\bullet$ for $\bullet = \rm{min}, \rm{med}, \rm{max}$.  

Since $|\Amed|<\infty$ and $\Lambda_P(\xi) \sim 1$ on $\Amed$, 
$$
\|\hat f(P(|\xi|),\xi)\|_{L^r(\Amed)} \lesssim \|f\|_p, \qquad 0 < r \leq \tfrac{dp'}{d+2}.
$$
This includes the range in the corollary, so it suffices to control the low and high frequency parts.  

For $\xi \in A_\bullet$, $\Lambda_P(\xi) \sim |\xi|^{\frac{(n_\bullet-2)d}{d+2}}$.  Thus by \eqref{E:cor 1} and the Lorentz space version of H\"older's inequality (\cite{SteinWeiss}),
$$
\|\hat f(P(|\xi|),\xi) |\xi|^{\frac{n_\bullet+d}{p'} - \frac dr}\|_{L^{r,q}(A_\bullet)} \lesssim \|f\|_p, \qquad 0 < r \leq q = \tfrac{dp'}{d+2}, \quad 1 \leq p < p_0.
$$
Performing Marcinkiewicz interpolation along segments with $\frac{n_\bullet+d}{p'}-\frac dr$ equal to a constant,
\begin{equation} \label{E:cor 2}
\|\hat f(P(|\xi|),\xi) |\xi|^{\frac{n_\bullet+d}{p'} - \frac dr}\|_{L^{r,p}(A_\bullet)} \lesssim \|f\|_p, \qquad 0 < r < \tfrac{dp'}{d+2}, \: 1 < p < p_0.
\end{equation}

Now we turn to the low frequency part.  By \eqref{E:cor 2}, 
$$
\|\hat f(P(|\xi|),\xi) \|_{L^{r,p}(\Amin)} \lesssim \|f\|_p, \qquad r = \tfrac{dp'}{\nmin+d}.
$$
When $r \geq p$, the left side bounds the $L^r(\Amin)$ norm, which in turn bounds the $L^s(\Amin)$ norm for all $s \leq r$, since $|\Amin| < \infty$.  If $r < p$, we set $A_k = \{|\xi| \sim 2^k\}$ and let $q = \tfrac{dp'}{d+2}$.  Then $q > r$, so by H\"older's inequality and \eqref{E:cor 1},
\begin{equation} \label{E:cor 3}
\begin{aligned}
\|\hat f(P(|\xi|),\xi)\|_{L^r(A_k)} &\lesssim 2^{kd(\frac1r-\frac1q)}2^{-k\frac{\nmin-2}{p'}}\|\hat f(P(|\xi|),\xi)\|_{L^q(A_k;\Lambda_P)}\\
& \lesssim 2^{k(\frac dr - \frac{\nmin+d}{p'})}\|f\|_p.
\end{aligned}
\end{equation}
For $r < \tfrac{dp'}{\nmin+d}$, we can sum over those $k$ such that $A_k \cap \Amin \neq \emptyset$, obtaining
$$
\|\hat f(P(|\xi|),\xi)\|_{L^r(\Amin)} \lesssim \|f\|_p.
$$

Now we turn to the high frequency terms.  Since $|\xi| \gtrsim 1$ on $\Amax$, \eqref{E:cor 2} implies that
$$
\|\hat f(P(|\xi|),\xi)\|_{L^{r,p}(\Amax)} \lesssim \|f\|_p, \qquad \tfrac{dp'}{\nmax+d} \leq r \leq \tfrac{dp'}{d+2}, \: 1 \leq p < p_0.
$$
If $r \geq p$, the left side of this inequality bounds the $L^r$ norm and we are done.  If $r < p$, we argue exactly as in \eqref{E:cor 3} to obtain
$$
\|\hat f(P(|\xi|),\xi)\|_{L^r(A_k)} \lesssim 2^{k(\frac dr - \frac{\nmax+d}{p'})}\|f\|_p,
$$
which is summable over large $k$ for $r > \tfrac{dp'}{\nmax+d}$.  
\end{proof}

The sharpness of the corollary in the case $r \geq p$ is known, and for $r < p$, it has a similar proof to the analogous result in \cite{BAJM}.  

We close with the essentially trivial deduction of uniform local estimates from elliptic restriction theorems off the scaling line.    Our motivations are two-fold.  First, this allows us to obtain bounds in the Bourgain--Guth range (\cite{BourgainGuth}).  Second, in the negatively curved case, no scaling-critical estimates are known beyond Stein--Tomas (\cite{VargasNeg, LeeNeg}), so these arguments may be helpful in a consideration of more general hypersurfaces.  

\begin{proposition}\label{P:off scaling}
Assume that $\scriptR^*(p \to q)$ holds for some $q$ greater than the maximum of $\tfrac{(d+2)p'}d$ and $\tfrac{2(d+1)}d$.  Then for all bounded sets $K \subseteq \R^d$ and even polynomials $P$ with non-negative coefficients,
$$
\|\Lambda_P(\nabla)^{1/\tilde p'}\scriptE_P(\chi_K f)\|_{L^{q}} \lesssim \||\xi|^\alpha f\|_{L^{p}}, \qquad \tilde p' := \tfrac{dq}{d+2}, \qquad  \alpha < \tfrac d{\tilde p}-\tfrac d{p}.
$$
The implicit constant depends on $K$, $\alpha$, and the degree of $P$.  
\end{proposition}

\begin{proof}
We may assume that $K = B(0,R)$ for some $R>0$.  Choose intervals $J_j$ as in Section~\ref{S:linear}:  so that $P(t) \sim a_j t^{2j}$ on $J_j$.  It suffices to prove uniform estimates over each annulus $A_j := \{\xi \in K : |\xi| \in J_j\}$.  For $k \in \Z$, $2^{-k} \leq 2R$, let 
$$
A_{jk} = \{\xi \in A_j : 2^{-k-1} \leq |\xi| \leq 2^{-k}\}.
$$
Rescaling $\scriptR^*(p \to q)$,
\begin{align*}
\|\Lambda_P(\nabla)^{1/\tilde p'} \scriptE_P(\chi_{A_{jk}}f)\|_{q} &\lesssim 2^{-kd+\tfrac{(d+2)k}{q}+\tfrac{dk}{p}}\|f\chi_{A_{jk}}\|_p\\
&\sim 2^{k(\alpha - (\frac d{\tilde p}-\frac d{p}))}\||\xi|^\alpha \chi_{A_{jk}} f\|_{L^{p}}.
\end{align*}
The right side is clearly summable, with bounds depending on $\alpha, R$.  
\end{proof}

There is the question of the endpoint $\alpha = \tfrac d{\tilde p}-\tfrac d{p}$.  When $q \geq p$, it is possible to deduce, using the methods of this article, conditional results, but the exponents are typically worse than those in Proposition~\ref{P:off scaling}.   If $q < p$, the endpoint is false.  This can be seen by considering functions of the form $f = \sum 2^{kd/\tilde p} e^{ix_k \cdot \xi} f_k$, with the $x_k$ sufficiently widely separated, $\supp f_k \subseteq \{2^{-k-1} < |\xi| < 2^{-k}\}$, and the $f_k$ quasi-extremal in the sense that 
$$
\|\Lambda_P(|\nabla|)^{1/\tilde p'} \scriptE_P f_k\|_q \gtrsim \||\nabla|^\alpha f_k\|_p \sim 2^{-kd/\tilde p}.  
$$

By way of comparison, a scaling critical adjoint restriction theorem for elliptic hypersurfaces, $\scriptR^*(p_0 \to q_0)$, would imply (by H\"older and Theorem~\ref{T:main}) that for any compact $K \subseteq \R^d$, $q > q_0$, $p \geq \tilde p := (\tfrac{(d+2)q}d)'$, and $\alpha < \tfrac d{\tilde p}-\tfrac dp$,
\begin{equation} \label{E:off line holder}
\|\Lambda_P(\nabla)^{\frac1{\tilde p'}} \scriptE_P f\|_q \lesssim \||\xi|^\alpha f\|_p, \qquad f \in L^q(\R^d)\, \: \supp f \subseteq K.
\end{equation}
If we instead use the Lorentz space version of H\"older's inequality and argue as in the proof of Corollary~\ref{C:unweighted}, we would have \eqref{E:off line holder} for all $q > q_0$, $q \geq p \geq \tilde p$ and $\alpha \leq \tfrac d{\tilde p} - \tfrac dp$.  In both cases, the implicit constants in \eqref{E:off line holder} depend on $q,p,K,\alpha,q_0$, and the degree of $P$.



\end{document}